\documentclass[12pt]{article}

\usepackage[english, main=english]{babel}
\usepackage{todonotes}
\usepackage{epsfig}
\usepackage{dsfont}
\usepackage{comment}
\usepackage{subfig}
\usepackage{amsmath,amssymb,amsthm,thmtools}
\usepackage{tikz}
\usepackage{xparse}
\usepackage{ifthen}
\usepackage{textcomp}
\usepackage{mathtools}
\usepackage{mathrsfs}
\usepackage{enumitem}
\usepackage{comment}
\usepackage[T1]{fontenc}
\usepackage[ansinew]{inputenc}
\usepackage{graphicx}
\usepackage{pdfpages}
\usepackage{float}
\RequirePackage{times} 
\usepackage[colorlinks,linkcolor=blue,urlcolor=black]{hyperref}
\usepackage[dvipsnames]{xcolor}
\allowdisplaybreaks

\newcommand{\N}{\mathbb{N}} 
\newcommand{\R}{\mathbb{R}} 
\newcommand{\C}{\mathbb{C}} 
\renewcommand{\P}{\mathbb{P}} 
\newcommand{\E}{\mathbb{E}} 
\newcommand{\1}{\mathds{1}}

\renewcommand{\d}{\mathrm{d}}

\theoremstyle{definition}
\newtheorem{definition}{Definition}[section]


\theoremstyle{plain}

\newtheorem{Lemma}[definition]{Lemma}
\newtheorem{Corollary}[definition]{Corollary}

\newtheorem{Theorem}[definition]{Theorem}

\theoremstyle{remark}
\newtheorem{Remark}{Remark}

\title{MA(1) processes with uniform innovations conditioned to stay positive in the non-expanding regime}
\author{Frank Aurzada and  Virginia Worf}

\begin{document}

\maketitle 
\allowdisplaybreaks

\begin{abstract}
We study an MA(1)-process with uniform innovations conditioned to stay positive. Representing the model as a Markov chain, we prove the existence of the limiting finite-dimensional distributions under this conditioning and identify the limiting process explicitly as a Doob $h$-transform. In the non-expanding case, i.e.\ when the coupling parameter $\theta$ satisfies $\theta\in[-1,1)$, we compute the relevant generating functions, extract sharp persistence asymptotics, and give explicit formulas for the eigenfunction $h$ and the persistence exponent. The resulting transition kernel of the limiting process is therefore fully explicit and displays a phase-dependent structure in the parameters. This provides a rare solvable example of a Markov chain on a continuous state space conditioned on persistence.
\end{abstract}


\section{Introduction and main results}
Stochastic processes conditioned to satisfy long-time constraints form a classical theme in probability, going back at least to Doob's seminal work on conditioning Brownian motion to be positive \cite{doob}. Conditional distributions which appear from constraining stochastic processes form interesting new models and give insights into the conditioning mechanism. In the present paper, we study this conditioning problem for a class of moving-average processes and give an explicit Doob-type description of the limiting process under the constraint of staying positive.

The present case is a rare instance where the quantities involved in the Doob-transform are all explicit but non-trivial. Indeed, although persistence probabilities for moving-average processes are already delicate (cf.\ the phase diagram in Figure~\ref{fig:map} below and \cite{AMZ,AurzadaRaschel}), we can go beyond exponential rates and identify the full limiting process under the conditioning on positivity. In particular, we obtain concrete formulas for the eigenfunctions entering the Doob $h$-transform. This gives a fairly complete picture of how the conditioned process depends on the parameters, including different phase regions. Beyond being new, our results can be seen as an expository example, where the non-trivial limiting Doob kernel is not only known to exist but can actually be written down explicitly.

More concretely, we consider the MA(1)-process  $Z_i\coloneqq X_{i+1}-\theta X_i$, $i\in\N$, $Z_0=X_1$, where $(X_i)_{i\in\N}$ are i.i.d.\ random variables, uniformly distributed on $[-a,1]$ for some $a>-1$ and $\theta\in[-1,1)$ is a coupling parameter. Since the choice of the positive right end-point of the uniform distribution at $1$ is without loss of generality, our results cover all possible uniform distributions with positive right end-point. The process $Z$ can be represented as a function of the Markov chain defined by $M_0:=(X_1,0)$ and  $M_i\coloneqq(X_{i+1},X_i)$ for $i\geq 1$. With $$S\coloneqq \{(x_2,x_1)\in [-a,1]^2| x_2>\theta x_1\}$$ and $\Omega_n\coloneqq \cap_{i=1}^n\{M_i\in S\}$, the persistence probability of the process $Z$ is given by
    \begin{align*}
    p_n^a(\theta)\coloneqq \P(\min_{1\leq i\leq n} Z_i\geq 0)=\P(X_2\geq\theta X_1,X_3\geq\theta X_2,\dots,X_{n+1}\geq\theta X_n)=\P(\Omega_n),~~~~n\ge 1,
\end{align*} with $p_0^a(\theta)\coloneqq 1$.
Let $\P_y(~\cdot~)\coloneqq\P(~\cdot~|X_1=y)$. Then the limiting finite dimensional distributions of $(Z_i)_{i\in\N}$ conditioned to stay positive and started at $Z_0=y$ are given by
\begin{align}\label{def_lim_process}
    \lim_{n\to\infty}\P_y((M_k)_{1\leq k\leq m}\in\cdot~|\Omega_n)=\P_y((M'_k)_{1\leq k\leq m}\in\cdot~),
\end{align}
provided that the process $M'$ on the right-hand side exists. 
Our result is the existence of a Markov chain $M'$ and an explicit formula for its transition kernel. Before stating the two main theorems formally, we introduce some notation.
\\The transition kernel of the Markov chain $M$ for $(x_1,x_1'),(x_2,x_2')\in[-a,1]^2$ is given by
\begin{align}\label{trans_prob_MC-2dim}
     \P_{(x_1,x_1')}(M_1\in (\d x_2,\d x_2'))=\delta_{x_1}(\d x_2') \, \P(X_1\in \d x_2)=\delta_{x_1}(\d x_2') \, \frac{1}{a+1}\,\d x_2.
\end{align} 
Note we only have to consider two points satisfying $x_2'=x_1$ and $x_1'$ is irrelevant. Therefore, instead of dealing with the transition kernel for two $\R^2$-valued points, it suffices to consider $x_1, x_2\in [-a,1]$ and define the reduced transition kernel of $M$ by
   \begin{align*}
       p(x_1,\d x_2)\coloneqq \frac{1}{a+1}\,\d x_2.
   \end{align*}

We are now ready to state the first main theorem. It shows that the limiting process in \eqref{def_lim_process} exists and has the expected Doob $h$-transform structure: its transition kernel is determined by a positive eigenfunction $h$ and the corresponding eigenvalue $\lambda$. The point of the second main result, Theorem~\ref{maintheorem}, is that this representation can be made fully explicit, with closed-form formulas for $h$ and $\lambda$ in all considered parameter regions.

\begin{Theorem}\label{maintheorem_2}
    For $\theta\in[-1,1)$ and $a>-1$, there is some $t\in\{a,-\frac{1}{\theta},0\}$ such that for $y\in[-t,1]$, the limiting process $M'$ in \eqref{def_lim_process} exists and is a time-homogeneous Markov chain on $[-t,1]$ with reduced transition kernel 
\begin{align*}
     p'(x_1,\d x_2)&\coloneqq \mathds{1}_S(x_2,x_1)\frac{h(x_2)}{h(x_1)}\frac{1}{\lambda}\frac{1}{a+1}\,\d x_2,~~~~~~~~x_1,x_2\in[-t,1],
\end{align*}
 where $h:\R\to\R$ is the unique (up to multiplication with a constant) and positive (on the effective state space) eigenfunction satisfying the eigenvalue equation
\begin{align}\label{eq_eigenfunction}
    \lambda h(x)=\int_{\{y>\theta x\}\cap[-t,1]} h(y) \frac{1}{a+1}\,\d y,\qquad x\in\R,
\end{align}
corresponding to the largest eigenvalue $\lambda>0$.
\end{Theorem}
\begin{Remark}
  Similar to the transition kernel $p$ of the Markov chain $M$, the actual transition kernel of $M'$ is $2$-dimensional. Due to the relation analogous to \eqref{trans_prob_MC-2dim}, it suffices to consider the reduced one-dimensional kernel $p'$.
\end{Remark}
\begin{Remark}
  For the same $\lambda$ as in \eqref{eq_eigenfunction}, consider the left eigenvalue equation 
  \begin{align}\label{stationary_distribution}
      \lambda\eta(y)=\int_{\{y>\theta x\}\cap[-t,1]}\eta(x)\frac{1}{a+1}\,\d x.
  \end{align} 
  Interpreting \eqref{eq_eigenfunction} as an equation of the non-zero, non-negative and compact operator $T$, which is the sub-stochastic kernel of the Markov chain $M'$ killed when leaving $S$, \eqref{stationary_distribution} is an equation of the associated adjoint operator $T^*$ which is also non-zero, non-negative and compact. Since the spectral radius of $T$ is $\lambda$, the spectral radius of $T^*$ is $\lambda$ as well and thus, by the Krein-Rutman theorem, see e.g.\ Theorem 19.2 in \cite{KreinRutman}, $\lambda$ is an eigenvalue of $T^*$. In particular, there is a non-negative and non-vanishing solution $\eta$ of \eqref{stationary_distribution} such that $\pi(x)\coloneqq h(x)\eta(x)$, $x\in[-t,1]$, is a probability density. The stationary distribution of the Markov chain $M'$ is then given by $\pi$.
\end{Remark}
The largest eigenvalue $\lambda$ in \eqref{eq_eigenfunction} corresponds to the decay rate of the persistence probabilities, also called persistence exponent, defined by 
$$\log(\lambda)\coloneqq\lim_{n\to\infty}\frac{1}{n}\log p_n^a(\theta),$$ see Section 2.1 in \cite{AMZ}.
Our second main contribution is to provide an explicit formula for $h$ in all parameter regimes. To do so, we introduce the following complex function: for $r,z\in\C$ with $|r|\leq1$, define the deformed exponential function by 
\begin{align*}
E(r,z)\coloneqq \sum_{n=0}^\infty \frac{r^{n(n-1)/2}}{n!}z^n.
\end{align*}

Now, the second main theorem makes the above Doob $h$-transform representation explicit by giving the eigenfunction $h$, uniquely determined up to multiplication by a constant, in each of the six parameter regimes for $a$ and the coupling parameter $\theta$, corresponding to one of the regions blue (B), green (G), yellow (Y), orange (O) and white (W) used in \cite{AurzadaRaschel} and turquoise (T), all presented in Figure~\ref{fig:map} below. 

\begin{Theorem}\label{maintheorem}
    For $\theta\in[-1,1)$, $a>-1$, $t=t(\theta,a)$ given in the different regions below and $y\in[-t,1]$, the limiting process $M'$ in \eqref{def_lim_process} exists and is a time-homogeneous Markov chain on $[-t,1]$ with transition kernel 
\begin{align*}
     p'(x_1,\d x_2)&\coloneqq\mathds{1}_S(x_2,x_1)\frac{h(x_2)}{h(x_1)}\frac{1}{\lambda}\frac{1}{a+1}\,\d x_2,~~~~~~~~x_1,x_2\in[-t,1],
\end{align*}
 where $\lambda$ is the persistence exponent and $h:\R\to\R$ has the following form:
 \begin{align*}
     h(x)=\begin{cases}
         0,~~~&1\le \theta x,
         \\\text{positive and region dependent},&-t\le\theta x<1,
         \\h\left(\frac{-t}{\theta}\right),~~&\theta x<-t,
     \end{cases}
 \end{align*}
 where we distinguish the different regions:
\begin{enumerate}[label={\rm (\Alph{*})},ref={\rm (\Alph{*})}]
    \setcounter{enumi}{1}
     \item  Let either $\theta\in(0,1)$ and $a\geq 0$ or $\theta\in[-1,0)$ and $a\in[-\theta,\frac{-1}{\theta}]$. Then $t=a$, $\lambda$ is the inverse of the zero with the smallest absolute value of the function $z\mapsto E\left(\theta,\frac{-z}{a+1}\right)$ and
        $$h(x)=E\left(\theta,\frac{-\theta x}{\lambda (a+1)}\right),~~~-a\leq\theta x<1.
            $$
\setcounter{enumi}{6}
     \item  Let $\theta\in[-1,0)$ and $a\geq \frac{-1}{\theta}$. Then $t=-\frac{1}{\theta}$, $\lambda$ is the inverse of the zero with the smallest absolute value of the function $z\mapsto E\left(\theta,\frac{-z}{a+1}\right)$ and
     $$h(x)= E\left(\theta,\frac{-\theta x}{\lambda (a+1)}\right),~~~\frac{1}{\theta }\leq \theta x< 1.$$
\setcounter{enumi}{24}
     \item Let $\theta\in[-1,0)$ and $a\in(0,-\theta)$. Then $t=a$, $\lambda$ is the inverse of the zero with the smallest absolute value of the equation $\frac{E\left(\theta,\frac{za}{\theta(a+1)}\right)}{E\left(\theta,\frac{za}{a+1}\right)}=z\cdot\frac{\theta +a}{\theta(a+1)}$ and
     $$h(x)=\begin{cases}E\left(\theta,\frac{-\theta x}{\lambda(a+1)}\right),&~-a \leq \theta x\leq  \frac{-a}{\theta},
     \\\frac{(1-\theta x)}{\lambda(a+1)}E\left(\theta,\frac{a}{\lambda(a+1)}\right),&~~\frac{-a}{\theta}< \theta x< 1.\end{cases}$$
    \setcounter{enumi}{14}
     \item Let $\theta\in(0,1)$ and $a\in[-\theta,0)$. Then $t=a$, $\lambda$ is the inverse of the zero with the smallest absolute value of the function $u$ defined below and
     $$h(x)=\begin{cases}
           \sum_{i=0}^{m+1}c_i(\theta,\lambda,a,m+1)x^i,&\!\!\!\theta x\in \bigg[\frac{-a}{\theta^m},\frac{-a}{\theta^{m+1}}\bigg),0\!\le \!m\!\le\! p\!-\!1,
            \\ \frac{1}{\lambda(a+1)}\sum_{i=0}^{p}c_i(\theta,\lambda,a,p)\frac{1}{i+1}\left(1-(\theta x)^{i+1}\right),&\!\!\!\theta x\in \bigg[\frac{-a}{\theta^p},1\bigg),
        \end{cases}$$
        where $p=p(\theta,a)\!\coloneqq\!\max\{n\in\N_0:\theta^n\geq\!-a\}$, $u(z)\coloneqq 1-z\sum_{i=0}^p\frac{(\theta^i+a)^{i+1}}{(i+1)!\theta^{\frac{i(i+1)}{2}}(a+1)^{i+1}}(-z)^i$,
                \begin{align} \label{eqn:orange:defnciwithlambda}
             c_i(\theta,\lambda,a,m)\!=\!\frac{-1}{u'\big(\frac{1}{\lambda}\big)}\!\left(\frac{-1}{\lambda(a\!+\!1)}\right)^i\!\frac{\theta^\frac{i(i+1)}{2}}{i!}\sum_{{\ell}=0}^{m-i}\!\!\!\!\!\sum_{\substack{i_r\ge 0\\i_1+\dots +i_{\ell}= m\!-i\!-{\ell}}}\!\!\!\!\!\prod_{r=1}^{\ell}s(m\!-\!i\!-\!(r\!-\!1)\!-\!\sum_{j=1}^{r-1}i_j,i_r),
        \end{align} 
        with the convention that the sum on the right-hand side is equal to $1$ for ${\ell}=0$, $i=m$ and equal to $0$ for ${\ell}=0$, $i<m$ and the abbreviation
        \begin{align}\label{def_orange_s}
            s(k,i)\coloneqq \theta^{\frac{i(i+1)}{2}-ki}\left(\frac{a}{\lambda(a+1)}\right)^{i+1}\frac{1}{(i+1)!}\left(\frac{\lambda(a+1)(i+1)}{a}-\theta^{i+1-k}\right).
        \end{align}
\setcounter{enumi}{22}
     \item  Let $\theta\in[-1,1)$ and $-1<a\leq\min(-\theta,0)$. Then $t=a$ and $h(x)=\lambda=1$.
\setcounter{enumi}{19}
    \item Let $\theta=0$ and $a\ge 0$. Then $t=0$, $h(x)=1$ and $\lambda=\frac{1}{a+1}$.
\end{enumerate}
\end{Theorem}
In the following, we will use the notation 
    \begin{align*}
    p_n^{a,y}(\theta)\coloneqq\P_{y}(\Omega_n)=\P(X_2\geq\theta y,X_3\geq\theta X_2,\dots,X_{n+1}\geq\theta X_n),~~~~n\ge 1,
    \end{align*}
    with $p_0^{a,y}(\theta)\coloneqq1$  for the persistence probability of the Markov chain $(M_i)$ staying in $S$ and started at $M_0=(y,0)$ with $y\in\R$. To find a Doob $h$-representation, we need the asymptotic behavior of $(p_n^{a,y}(\theta))$ as $n\to\infty$. This will be deduced from its generating function and the generating function of the $(p_n(\theta))$ defined by 
    $$\hat{P}_{a,y,\theta}(z)\coloneqq \sum_{n=0}^{\infty}p_n^{a,y}(\theta)z^n,\qquad \hat{P}_{a,\theta}(z)\coloneqq \sum_{n=0}^{\infty}p_n^{a}(\theta)z^n.$$
\begin{figure}[H]
\begin{center}
\begin{tikzpicture}[scale=0.8]
\filldraw[fill=blue!50!white, draw=blue!50!white] (-4,4) -- (0,0) -- (0,4) -- cycle; 
\filldraw[fill=blue!50!white, draw=blue!50!white] (0,4) -- (0,0) -- (4,0) -- (4,4) -- cycle; 
\filldraw[fill=blue!50!white, draw=blue!50!white] (4,4) -- (4,8) -- (-2.1,8)  -- (-2.1,4) -- cycle; 

\fill[blue!50!white]  plot[domain=-4:-2, samples=100] (\x, {-16/\x})   -- (-2,4) -- (-4,4) -- cycle;

\fill[green!60!white]
  plot[domain=-4:-2, samples=100] (\x, {-16/\x})
  -- (-2, 8.01) -- (-4, 8.01) -- cycle;

\filldraw[fill=yellow!60!white, draw=yellow!60!white] (-4,4) -- (0,0) -- (-4,0) -- cycle;

\filldraw[fill=orange!80!white, draw=orange!80!white] (0,0) -- (4,-4) -- (4,0) -- cycle;

\draw[->,thick] (-3.5,2) -- (-3.5,6);

\node [black] at (-2,-2) {$p_n^{a,y}(\theta)\equiv 1$};
\node [black] at (1.5,-3.3) {$p_n^{a,y}(\theta)\equiv 1$};

\draw[->,thick] (-3.5,7) -- (-2.6,6.15);
\draw[-,dashed] (4,4) -- (-4,4);

\draw[->,very thick] (-4,0) -- (4,0);
\node [black] at (-3,4.4) {\eqref{generatingfunction_yellow_and_green_rel}};
\draw[-,very thick] (0,-4) -- (0,0);
\draw[->, RoyalBlue, very thick] (0,0) -- (0,9);
\node [black] at (-2.4,6.6) {\eqref{reduction_green_blue_prob}};
\node [black] at (4.3,0) {\large $\theta$};
\node [black] at (4,-0.5) {\large $1$};
\node [black] at (-4,-0.5) {\large $-1$};
\node [black] at (0,9.5) {\large $a$};
\node [black] at (0.3,4.4) {\large $1$};
\node [black] at (0.5,-0.5) {\large $0$};
\node [black] at (-0.7,-3.7) {\large $-1$};

\end{tikzpicture}
\end{center}
\caption{Phase diagram of the main results}
\label{fig:map}
\end{figure}
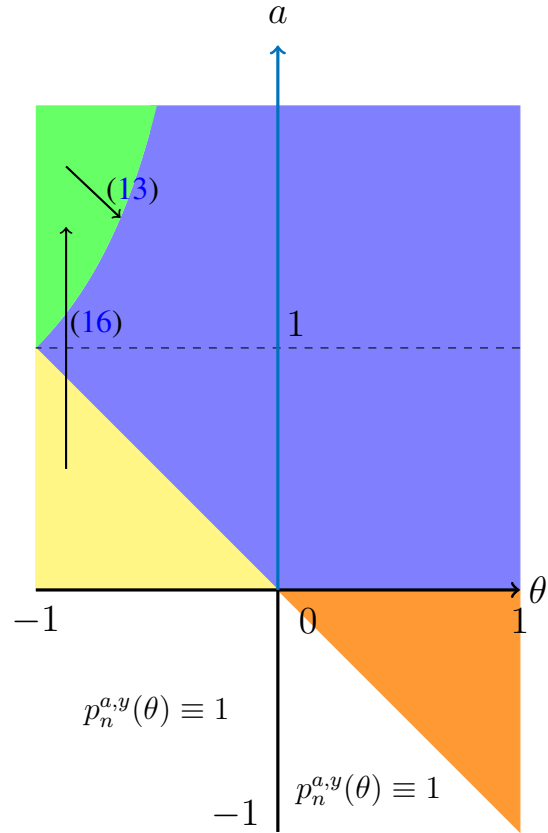
\begin{Remark}
  The form of $h$ in the trivial cases $\theta x\ge1$ and $\theta x<-t$ follows directly from equation \eqref{eq_eigenfunction} and the continuity of $h$.
\end{Remark}
\begin{Remark}
Note that in the blue region (B) for the special case $a=-1/\theta$ as well as the for the entire green region (G), the point $x=1/\theta$ is not part of the state space of the conditioned Markov chain, because starting there makes the chain violate the conditioning even in the first step. We did not mark this in the theorems for the sake of readability; in fact, the state space should be read as $(-\frac{1}{\theta},1]$ rather than $[-\frac{1}{\theta},1]$.
\end{Remark}
\begin{Remark}
   In the white region (W), $p_n^{a,y}(\theta)\equiv 1$, so Theorems~\ref{maintheorem_2} and~\ref{maintheorem} follow directly.
\end{Remark}
\begin{Remark}
    In the turquoise region (T), the result follows since $(x_2,x_1)\in S$ is equivalent to $x_2\ge0$. 
\end{Remark}
\begin{Remark}
In the blue region (B), part of the state space $[-a,1]$ is transient. In fact, for $\theta\in(0,1)$, we have $X_i\geq -a \theta^{i-1}$ for all $i\geq 2$, under the conditioning. Similarly, for $\theta\in[-1,0)$, $X_i \geq \theta$ for $i\geq 2$.  A similar phenomenon appears in the green region (G): $X_i\geq \theta$ for $i\geq 2$ under the conditioning.
\end{Remark}

\begin{Remark}
    The case $\theta=1$ is excluded in Theorem~\ref{maintheorem} since the conditioning in \eqref{def_lim_process} no longer produces a non-trivial Doob-type limit. Indeed, consider exemplarily the following one-dimensional distribution: for $x>y$, we have $\P(M_1\in[x,1]\times\{y\}|\Omega_n) = \frac{\P( x\leq X_2\leq X_3\leq \ldots\leq X_{n+1})}{\P( y\leq X_2\leq X_3\leq \ldots \leq X_{n+1})}\to 0$, as $n\to\infty$. In fact, the limit in \eqref{def_lim_process} only leads to the degenerate Markov chain $y=X_2=X_3=\ldots$. On the level of persistence probabilities, this corresponds to the fact that for $y\in[-a,1]$, $p_n^{a,y}(1)=\frac{1}{n!}\left(\frac{1-y}{a+1}\right)^n$. Similarly, or equivalently, we have $p_n^a(1)=\frac{1}{n!}$, a formula that first appeared in \cite{12,15} for $a=1$.
\end{Remark}
\paragraph{Related work.}
The starting point of our work is the setting from \cite{AurzadaRaschel}, where explicit expressions for the generating function of the persistence probabilities of a moving average process with uniform innovations and without fixed starting point in the different regions are derived. In particular, the results obtained in \cite{AurzadaRaschel} for $|\theta|<1$ follow from our results by integrating out the starting point $y$. Our approach for calculating the generating functions builds on similar techniques.
\\General Markov chains and their non-exit probabilities from subsets of the state space are studied in \cite{AMZ}. In particular, the eigenvalue equation and the associated persistence exponent of moving average processes with uniform innovations used in the present paper are analyzed.
\\ The Doob transformation is a well-studied technique in probability and has been used in many contexts. A particularly nice exposition of the case of Markov chains on a countable state space conditioned to stay inside a set and started from a fixed point is provided in \cite{Swart} (see Theorem 2.11). This resembles the main result of the present paper, which establishes an analogous representation for a special Markov chain on a continuous and compact state space. Related results using the Doob-transform for random walks can be found in the fundamental paper \cite{BertoinDoney}. Further closely related results are \cite{Wachtel1,Wachtel4,Wachtel6,Wachtel7,Wachtel2,Wachtel3,EichelsbacherKonig}. The framework of Doob-transforms was first introduced in the classical work \cite{doob}. General theory on Doob-transforms can be found in \cite{RogersWilliams1} with more detailed work in the setting of Brownian motion provided in \cite{RogersWilliams2}. Another fundamental reference is \cite{ColletMartinezSanMartin}, see also \cite{ChampagnatVillemonais2016, SenetaVereJones}.
\\Concerning persistence probabilities, we also mention \cite{aurzada2025persistenceprobabilitiesma1sequences}, where moving average processes with Laplace innovations are studied, and \cite{2}, which investigates MA(1)-processes with standard normally distributed innovations using functional analytic methods. Further estimates in the Gaussian setting can be found in \cite{12}. The study of persistence for MA(1)-type stationary sequences goes back, in this context, to Majumdar and Dhar \cite{15}, who derived nonlocal generating-function equations for such persistence probabilities. Closely related problems for random walks in cones are obtained in \cite{Wachtel5}.
\\ Related questions for autoregressive processes are addressed in \cite{1,3,5,9,DonnartSimon,11,14}.
\\A motivation for studying persistence from a theoretical physics point of view as well as a survey of that perspective is provided in  \cite{8}. The angle on persistence probabilities is to study how quickly a physical system initialized out of equilibrium returns to equilibrium.
\paragraph{Extensions.} In the present work, we restrict to the non-expanding regime $|\theta|\le1$. In \cite{AurzadaRaschel}, the generating function of the $(p_n^a(\theta))$ is computed also for the expanding case $|\theta|>1$.
However, the dualities used there are not applicable for fixed starting points. Nonetheless, our techniques can be extended to the case $|\theta|>1$, but since this is technically more involved, an extension to this parameter region will be postponed to future work. Another possible extension of the presented techniques are MA(1)-processes with different  innovations and autoregressive processes.
\paragraph{Outline.}
The paper is organized as follows. In the next section, we will prove some technical lemmas.  The proof of Theorem~\ref{maintheorem} is carried out in the subsequent five Sections~\ref{blue_region}-\ref{orange_region}, assigned to the respective regions. In each region, we begin by computing the generating function of the persistence probabilities $\left(p_n^{a,y}(\theta)\right)$, deduce their asymptotic behavior and conclude by finding explicit formulas for the eigenfunction $h$ and eigenvalue $\lambda$. By showing that the concrete function $h$ from Theorem~\ref{maintheorem} satisfies the eigenvalue equation \eqref{eq_eigenfunction} with eigenvalue $\lambda$, we immediately prove Theorem~\ref{maintheorem_2} as well.
\section{Preliminaries}
For the persistence probability of the process started from a fixed point, we have the following estimates, the proofs of which are immediate from the definition.
\begin{Lemma}\label{P_x_estimate}
    Let $x\in[-a,1]$. Then 
    \begin{align*}
        \P_x(\Omega_n)\leq\P_{-a}(\Omega_n),\quad \text{for $\theta\geq 0$},\qquad \text{and }\qquad  \P_x(\Omega_n)\leq\P_{1}(\Omega_n),\quad \text{for $\theta\leq 0$}.
    \end{align*}
\end{Lemma}

These estimates justify the application of the dominated convergence theorem in the following lemmas for exchanging limits and sums. The next lemma shows that the asymptotics of the persistence probability implies the eigenvalue equation.

\begin{Lemma}\label{Prob=Lim_int_ef}
  Assume that there are constants $\lambda=\lambda(\theta,a)>0$ and $c(\theta,a)>0$ not depending on $x$ or $n$ and a continuous function $h$ such that
  $$p_n^{a,x}(\theta)\sim c(\theta,a)h(x)\lambda^{n+1}$$
  holds as $n\to\infty$ for all $x\in[-a,1]$. Then,
  \begin{align} \label{eqn:ppimplyeveqn}
    \lambda h(x) & =\lim_{n\to\infty}\frac{1}{c(\theta,a)\lambda^{n}}\, p_n^{a,x}(\theta) =\int_{\{y>\theta x\}\cap[-a,1]}h(y)\,\frac{1}{a+1}\,\d y.
  \end{align}
  \end{Lemma}
    \begin{proof}
  Due to the Markov property of $(M_i)_{i\in\N}$ with transition kernel $p$:
\begin{align*}
     \lim_{n\to\infty}\frac{1}{c(\theta,a)\lambda^{n}} \, p_n^{a,x}(\theta)
      &=\lim_{n\to\infty}\frac{1}{c(\theta,a)\lambda^{n}}\,\P_x(M_1\in S,\dots,M_n\in S)
      \\&=\lim_{n\to\infty}\frac{1}{c(\theta,a)\lambda^{n}}\int\P_y(M_1\in S,\dots M_{n-1}\in S)\1_{\{(y,x)\in S\}}p(x,\d y)
      \\&=\lim_{n\to\infty}\int_{\{y>\theta x\}\cap[-a,1]}\frac{\P_y(M_1\in S,\dots M_{n-1}\in S)}{c(\theta,a)\lambda^{n}}\,\frac{1}{a+1}\,\d y
      \\&\overset{ }{=}\int_{\{y>\theta x\}\cap[-a,1]}\lim_{n\to\infty}\frac{\P_y(M_1\in S,\dots M_{n-1}\in S)}{c(\theta,a)\lambda^{n}}\frac{1}{a+1}\,\d y
      \\&=\int_{\{y>\theta x\}\cap[-a,1]}\lim_{n\to\infty}\frac{p_{n-1}^{a,y}(\theta)}{c(\theta,a)\lambda^{n}}\frac{1}{a+1}\,\d y
      \\&=\int_{\{y>\theta x\}\cap[-a,1]}h(y)\,\frac{1}{a+1}\,\d y,
\end{align*}
where we used Lemma~\ref{P_x_estimate} to find an integrable majorant for the exchange of limit and integral and we used the assumption in the last step. 
\end{proof}
The next lemma contains the core of the classical Doob transformation and will be applied in order to show Theorem~\ref{maintheorem}.
\begin{Lemma}\label{Prob=Lim_int_trans_kernel}
Assume that
  $$p_n^{a,x}(\theta)\sim c(\theta,a)h(x)\lambda^{n+1}$$
   holds as $n\to\infty$ for all $x\in[-a,1]$, a continuous function $h$ which is strictly positive on $[-a,1]$ and constants $\lambda=\lambda(\theta,a)>0$, $c(\theta,a)>0$ not depending on $x$ or $n$.
Then, for any $y\in[-a,1]$, the limit
\begin{align*}
    \lim_{n\to\infty}\P_y((M_k)_{1\leq k\leq m}\in\cdot~|\Omega_n)=\P_y((M'_k)_{1\leq k\leq m}\in\cdot~)
\end{align*}
exists, and $M'$ is a Markov chain with reduced transition kernel 
\begin{align} \label{eqn:generalkernel2}
    p'(x_1,\d x_2)\coloneqq \mathds{1}_S(x_2,x_1)\frac{h(x_2)}{h(x_{1})}\frac{1}{\lambda}\frac{1}{a+1}\,\d x_2,~~~~~~~~x_1,x_2\in[-a,1].
\end{align}
\end{Lemma}
\begin{proof}
By Lemma~\ref{Prob=Lim_int_ef} and the assumption, the function $h$ satisfies equation \eqref{eqn:ppimplyeveqn}. 
    \\Let $C_1\dots , C_m\in\mathcal{B}(\R)$, $\mathbf{x}\coloneqq(x_1,\dots,x_m)\in [-a,1]^m$ and set $x_0\coloneqq y$. For fixed $m\in\N$, we have, by the Markov property,
\begin{align*}
    &\lim_{n\to\infty}\P_{y}(M_1\in C_1\times\{y\},M_2\in C_2\times C_1,\dots,M_m\in C_m\times C_{m-1}~|\Omega_n)
    \\
    &=\lim_{n\to\infty}\int_{[-a,1]^m}\1_{C_1\times\dots\times C_m}(\mathbf{x})\left[\prod_{i=1}^m p(x_{i-1},\d x_i)\mathds{1}_S(x_i,x_{i-1})\right]\frac{\P_{x_m}(\Omega_{n-m})}{\P_{x_0}(\Omega_n)} 
    \\&=\int_{[-a,1]^m}\1_{C_1\times\dots\times C_m}(\mathbf{x})\left[\prod_{i=1}^m p(x_{i-1},\d x_i)\mathds{1}_S(x_i,x_{i-1})\right]\lim_{n\to\infty}\frac{\P_{x_m}(\Omega_{n-m})}{\P_{x_0}(\Omega_n)}
   \\&=\int_{[-a,1]^m}\1_{C_1\times\dots\times C_m}(\mathbf{x})\left[\prod_{i=1}^m \frac{1}{a+1}\,\d x_i\mathds{1}_S(x_i,x_{i-1})\right]\frac{h(x_m)}{h(x_0)}\lambda^{-m}
    \\&=\int_{[-a,1]^m}\1_{C_1\times\dots\times C_m}(\mathbf{x})\left[\prod_{i=1}^m \mathds{1}_S(x_i,x_{i-1})\frac{h(x_i)}{h(x_{i-1})}\frac{1}{\lambda}\frac{1}{a+1}\,\d x_i\right],
\end{align*}
where we used Lemma~\ref{P_x_estimate} to find an integrable majorant for the exchange of limit and integral and we used the assumption on the asymptotics of $(p_n^{a,x}(\theta))$ in the third step.
Thus, the process with density given by (\ref{eqn:generalkernel2}) is the limiting process $M'$.
One can check that the density is a Markov kernel using the eigenvalue equation \eqref{eqn:ppimplyeveqn}.
    \end{proof}
Next, we mention some properties of the deformed exponential function. It is well-known that it satisfies the following differential equations.
\begin{Lemma}
  For any $a,b,x\in\C$ with $|a|\le1$ and $n\geq 1$ we have
\begin{align}\label{eqn:dgl_n_exponential}
    \frac{\partial^n}{\partial x^n}E(a,bx)=a^{\frac{n(n-1)}{2}}b^nE(a,a^nbx).
\end{align}
\end{Lemma}

The next lemma deals with the zeros of the deformed exponential function. For $\theta\in(0,1)$, this result can be found in Theorem~6 and p.\ 320--332 of \cite{Morris}. For $\theta\in[-1,0)$, we refer to Appendix C of \cite{dyachenko}.

\begin{Lemma}\label{nullstellen_exponentialfunktion}
For $\theta\in(0,1)$, the zeros of the deformed exponential function $E(\theta,z)$ are simple, real, isolated and negative. For $\theta\in[-1,0]$, the zeros of the deformed exponential function $E(\theta,z)$ are real, the zero $z_0$ with the smallest absolute value is simple and negative and for all other zeros $z$, it holds $|\theta z|>|z_0|$.
\end{Lemma}

We will examine a number of power series and deduce the required asymptotics, so the following result is essential. The proof is standard and uses the power series representation of analytic functions as well as comparison of coefficients.
\begin{Lemma}\label{asymptotic_powerseries}
    Let $(p_n)$ be a sequence of real numbers. Let $z_0>0$. Assume that
\begin{equation}
\sum_{n=0}^\infty p_n z^n = \frac{f(z)}{z_0-z},
\end{equation}
for complex $|z| < z_0$. Further assume that $f$ is analytic on $|z|<z_1$ with $z_0<z_1$ and that $f(z_0)\neq 0$. Then
$$
p_n \sim f(z_0) \cdot z_0^{-(n+1)},\qquad \text{as $n\to\infty$.}
$$
\end{Lemma}
\section{The blue region}\label{blue_region}
Recall that in the blue region, $\theta\in(0,1)$ and $a\geq 0$ or $\theta\in[-1,0)$ and   $a\in[-\theta,\frac{-1}{\theta}]$. 
A major role in this section will be played by the value $\lambda$, which is given by the inverse $\lambda=\frac{1}{z_0}$ of the smallest positive $z_0$ of the equation
\begin{align}\label{root_blue}
    E\left(\theta,\frac{-z}{a+1}\right)=0.
\end{align}
Note that $z_0$ is the complex root with smallest modulus, by Lemma~\ref{nullstellen_exponentialfunktion}. We will see that $\lambda$ corresponds to the persistence exponent.
With this notation, we can compute the generating function of $(p_n^{a,y}(\theta))$ explicitly.
\begin{Lemma}\label{generatingfunction_blaue_region}
   In the blue region, for $y\in[-a,1]$,we have $\hat{P}_{a,y,\theta}(z)=\frac{E\left(\theta,\frac{-\theta zy}{a+1}\right)}{E\left(\theta,\frac{-z}{a+1}\right)}$ with radius of convergence $\frac{1}{\lambda}$.
\end{Lemma}
    \begin{proof}
        From a recursion of the probabilities $(p_n^{a,y}(\theta))$, we will deduce a differential equation for the generating function which has the unique solution $\frac{E\left(\theta,\frac{-\theta zy}{a+1}\right)}{E\left(\theta,\frac{-z}{a+1}\right)}$. \\
        By the choice of the parameters in the blue region, we get that $y\in[-a,1]$ implies $\theta y\in[-a,1]$. Consider $y$ such that $\theta y\in[-a,1]$. For all $n\geq 1$, it holds
        \begin{align*}
            p_n^{a,y}(\theta)&=\frac{1}{(a+1)^n}\int_{\theta y}^1\int_{\theta x_2}^1\dots\int_{\theta x_n}^1 \d x_{n+1}\dots \d x_3\d x_2\\&=\frac{1}{a+1}\int_{\theta y}^1\left[\frac{1}{(a+1)^{n-1}}\int_{\theta x_2}^1\dots\int_{\theta x_n}^1 \d x_{n+1}\dots \d x_3\right]\d x_2=\frac{1}{a+1}\int_{\theta y}^1p_{n-1}^{a,x_2}(\theta)\d x_2.
        \end{align*}
        Multiplying with $z^n$, summing over $n$ and  taking the derivative with respect to $y$ gives 
    \begin{align*}
        \frac{\partial}{\partial y}\hat{P}_{a,y,\theta}(z)\notag
        &=\frac{\partial}{\partial y}\sum_{n=1}^{\infty} p_n^{a,y}(\theta)\cdot z^n
        = \frac{1}{a+1}~ \frac{\partial}{\partial y}\int_{\theta y}^1\left(\sum_{n=1}^{\infty} p_{n-1}^{a,x_2}(\theta)\cdot z^n\right) \d x_2\notag
     \\
     & =\frac{z}{a+1}~ \frac{\partial}{\partial y}\int_{\theta y}^1\hat{P}_{a,x_2,\theta}(z) \d x_2 \notag
      =\frac{-\theta z}{a+1}\hat{P}_{a,\theta y,\theta}(z),
    \end{align*}
    where we used that $p_0^{a,y}(\theta) z^0=1$ is constant in the first step and the fact that the integral is bounded and the series is absolutely convergent within the disk of convergence so that we can exchange the sum and the integral.
    In conclusion, we found a differential equation for the function $\hat{P}_{a,y,\theta}(z)$
    with the additional constraint $\hat{P}_{a,\frac{1}{\theta},\theta}(z)=1$ (coming from the definition of the generating function and $p_0^{a,y}(\theta)=1$ and $p_n^{a,\frac{1}{\theta}}(\theta)=0$ for $n\ge 1$).
    The unique solution to this differential equation with the boundary condition is $y\mapsto\frac{E\left(\theta,\frac{-\theta zy}{a+1}\right)}{E\left(\theta,\frac{-z}{a+1}\right)}$, as can be checked by taking derivatives (applying \eqref{eqn:dgl_n_exponential}) and checking the additional constraint. Since $\frac{1}{\lambda}$ is the unique zero with smallest modulus of the denominator and the numerator has no zero in $\frac{1}{\lambda}$ and no singularities, the radius of convergence is given by $\frac{1}{\lambda}$.
 \end{proof}    
With the concrete formula for $\hat{P}_{a,y,\theta}(z)$, we infer the following asymptotics from Lemma~\ref{asymptotic_powerseries}:
\begin{Lemma}\label{asymptotic_blau_erste}
In the blue region, as $n\to\infty$, it holds for $y\in[-a,1]$ that
    \begin{align*}
       p_n^{a,y}(\theta)\sim \frac{a+1}{E\left(\theta,\frac{-\theta }{\lambda(a+1)}\right)}\, \lambda^{n+1}E\left(\theta,\frac{-\theta y}{\lambda(a+1)}\right),
    \end{align*}
    where $\lambda=\lambda(\theta,a)$ is given as the smallest real positive root equation of \eqref{root_blue}.
\end{Lemma}
    \begin{proof} 
   We want to use Lemma~\ref{asymptotic_powerseries}, so we need to check that there is a function $f(z)$ with
    \begin{align}\label{asymptotic_blue_proof}
        \sum_{n=0}^{\infty}p_n^{a,y}(\theta)z^n=\frac{f(z)}{z_0-z}~~~~\text{for all}~|z|<z_0=\frac{1}{\lambda}
    \end{align}
   which is analytic on $|z|<z_1$ with some $z_1>\frac{1}{\lambda}$ and $f\left(z_0\right)\neq 0$.
   \\Recall that  $\sum_{n=0}^{\infty}p_n^{a,y}(\theta)z^n=\frac{E\left(\theta,\frac{-\theta zy}{a+1}\right)}{E\left(\theta,\frac{-z}{a+1}\right)}$. In order to find $f$ such that \eqref{asymptotic_blue_proof} holds, note that $z_0$ is the zero with smallest modulus of the denominator $g(z)$ of the generating function. We also know from Lemma~\ref{nullstellen_exponentialfunktion} that $z_0$ is a simple zero of $g$ and that the zero with the next-largest modulus has modulus $z_1>z_0$; so that $g(z)/(z_0-z)$ is analytic on $|z|<z_1$. Therefore $$u(z)\coloneqq E\left(\theta,\frac{-\theta zy}{a+1}\right)~~\text{ and }~~v(z)\coloneqq g(z)\frac{1}{z_0-z},$$ are analytic on $|z|<z_1$. We rewrite $\sum_{n=0}^{\infty}p_n^{a,y}(\theta)z^n=\frac{u(z)}{v(z)}\frac{1}{z_0-z}$. In order to show that the function $f(z)\coloneqq\frac{u(z)}{v(z)}$ satisfies the assumptions of Lemma~\ref{asymptotic_powerseries} it remains to be seen that $f(z_0)\neq 0$.
   Let us first compute
   \begin{align*}
   v(z_0)=-g'(z_0)=-E\left(\theta,\frac{-\theta z_0}{a+1}\right)\frac{-1}{a+1} = \frac{E\left(\theta,\frac{-\theta z_0}{a+1}\right)}{a+1},
   \end{align*}
   which is non-zero: Indeed, if $\theta>0$, we have $0<\theta z_0<z_0$ and  $z_0$ is the smallest positive zero of $z\mapsto E\left(\theta,\frac{-z}{a+1}\right)$. If $\theta<0$, we have $|\theta z_0|\le z_0$ and $\theta z_0\neq z_0$, so $\theta z_0$ is no zero of $z\mapsto E\left(\theta,\frac{-z}{a+1}\right)$.
   In particular, this gives
   \begin{align*}
       f(z_0)=\frac{a+1}{E\left(\theta,\frac{-\theta z_0}{a+1}\right)}E\left(\theta,\frac{-\theta yz_0}{a+1}\right)\neq 0,
   \end{align*}
   where $E\left(\theta,\frac{-\theta yz_0}{a+1}\right)\neq 0$: If $\theta>0$, $\theta y z_0$ is negative for $y<0$ and thus cannot be a root. If $y\in[0,1]$, we have $0<\theta y z_0<z_0$. For $\theta<0$, it holds $|\theta y z_0|\le z_0$ with $\theta y z_0\neq z_0$ by choice of $a$.
   An application of Lemma~\ref{asymptotic_powerseries} yields the claim.
    \end{proof}
Now we can deduce the statement of Theorem~\ref{maintheorem} for the blue region from Lemma~\ref{Prob=Lim_int_trans_kernel}. Note that Theorem~\ref{maintheorem_2} then follows directly by Lemma \ref{Prob=Lim_int_ef} and Lemma~\ref{asymptotic_blau_erste}.
\begin{proof}[Proof of Theorem~\ref{maintheorem} (B)]
For $y\in[-a,1]$, set $c(\theta,a)\coloneqq \frac{a+1}{E\left(\theta,\frac{-\theta }{\lambda(a+1)}\right)}$. By Lemma~\ref{asymptotic_blau_erste}, we have $p_n^{a,y}(\theta)\sim c(\theta,a)h(y)\lambda^{n+1}$ as $n\to\infty$. Lemma~\ref{Prob=Lim_int_trans_kernel} then yields the claim and the formula for $h(y)$ for $y\in[-a,1]$. The claim about the formula for $h(y)$ for the extended range $\theta y \in [-a,1]$ follows from the eigenvalue equation. 
    \end{proof}
\section{The green region}
The green region is characterized by  $\theta\in[-1,0)$ and $a\geq \frac{-1}{\theta}$. 
 In this setting, we follow the approach of \cite{AurzadaRaschel} and reduce the results in the green region to the results in the blue region. This allows us to directly apply the formula for the transition kernel of the limiting process in the blue region to obtain the transition kernel in the green region.  Nevertheless, we will still derive the generating function of $(p_n^{a,y}(\theta))$ in the green region since its explicit form will be needed later in the yellow region.
 \\In this region, denote by $\tilde{\lambda}=\lambda\left(\theta,-\frac{1}{\theta}\right)$ the persistence exponent from the blue region given as the solution of \eqref{root_blue} for the parameters $\theta$ and $\tilde a:=-1/\theta$.
\begin{Lemma}\label{generatingfunction_grune_region}
    In the green region, we have
\begin{align*}
 \hat{P}_{a,y,\theta}(z)=
    \begin{cases}
            1,~~&~-a\leq y<\frac{1}{\theta},
            \\\frac{E\left(\theta,\frac{-\theta zy}{a+1}\right)}{E\left(\theta,\frac{-z}{a+1}\right)},~~&~\frac{1}{\theta}\le y \leq 1,
        \end{cases}      
\end{align*}
    with radius of convergence $\frac{1}{\lambda}$ for $y>\frac{1}{\theta}$ and $\infty$ for $y\le\frac{1}{\theta}$, where $\lambda$ is the inverse of the root $z_0>0$ with the smallest absolute value of the equation $E\left(\theta,\frac{-z}{a+1}\right)=0$.
    \end{Lemma}
    \begin{proof}
    For $y\le\frac{1}{\theta}$, we use $X_2\sim U[-a,1]$ and get for $n\geq 1$:
     \begin{align*}
        p_n^{a,y}(\theta)\leq\P(X_2\ge1 ,X_3\geq\theta X_2,\dots,X_{n+1}\geq\theta X_n)=0.
    \end{align*}
    Therefore, we have $\hat{P}_{a,y,\theta}(z)=p_0^{a,y}(\theta)z^0=1$ for $y\in\left[-a,\frac{1}{\theta}\right]$.
    \\The case $y>\frac{1}{\theta}$ is left. Here, we reduce the case to the shared borderline with the blue region, which is $\theta\in[-1,0)$ and $\tilde a=\frac{-1}{\theta}$. To do so, we condition the occurring random variables to be $\geq\frac{1}{\theta}$.
    First, set $A\coloneqq \{X_i\geq\frac{1}{\theta}~\forall~2\leq i\leq n+1\}$. The conditions $\theta X_i\leq X_{i+1}\leq 1$ for all $2\leq i\leq n$ and $X_{n+1}\geq\theta X_n\geq \theta\geq\frac{1}{\theta}$ imply $X_i\geq\frac{1}{\theta}$ for all $2\leq i \leq n+1$, so we see that
    \begin{align*}
        \P(\{X_2\geq\theta y,X_3\geq\theta X_2,\dots,X_{n+1}\geq\theta X_n\}\cap A^c)=0.
    \end{align*}
    Now we use $\tilde{X}_i \sim U\left[\frac{1}{\theta},1\right]$ and $\nu\coloneqq\frac{1-\frac{1}{\theta}}{a+1}$ to write
    \begin{align}\label{reduction_green_blue_prob}
        p_n^{a,y}(\theta)&=\P(X_2\geq\theta y,X_3\geq\theta X_2,\dots,X_{n+1}\geq\theta X_n~|~A)\ \P(A)\notag
        \\&=\P(\tilde{X}_2\geq\theta y,\tilde{X}_3\geq\theta \tilde{X}_2,\dots,\tilde{X}_{n+1}\geq\theta \tilde{X}_n)\ \P\left(X_1\geq\frac{1}{\theta}\right)^n\notag
        \\&=\P(\tilde{X}_2\geq\theta y,\tilde{X}_3\geq\theta \tilde{X}_2,\dots,\tilde{X}_{n+1}\geq\theta \tilde{X}_n)\left(\frac{1-\frac{1}{\theta}}{a+1}\right)^n\notag
        \\&= p_n^{\frac{-1}{\theta},y}(\theta)\nu^n
    \end{align}
    and thus
    \begin{align*}
     \hat{P}_{a,y,\theta}(z)&\notag
     = \sum_{n=0}^{\infty}p_n^{\frac{-1}{\theta},y}(\theta)\left(\nu z\right)^n
     =\hat{P}_{\frac{-1}{\theta},y,\theta}\left(\nu z\right).
    \end{align*}
    Note that every starting point $\frac{1}{\theta}< y\leq 1$ from the green region is also valid in the blue region for $\theta\in[-1,0)$ and $\tilde a=\frac{-1}{\theta}$, so we already computed $\hat{P}_{\frac{-1}{\theta},y,\theta}(\cdot)$ in Lemma~\ref{generatingfunction_blaue_region} for the blue region. Plugging this in, we obtain the claim.
    \end{proof}
Let us record one immediate consequence of the last proof (i.e.\ from \eqref{reduction_green_blue_prob}).

\begin{Corollary}\label{pers_exp_green}
  In the green region, the persistence exponent is given by $\lambda=\tilde\lambda\frac{1-\frac{1}{\theta}}{a+1}$.
\end{Corollary}
    
Recall that for $y\le\frac{1}{\theta}$, we have 
     \begin{align*}
        \P_y(\Omega_n)=p_n^{a,y}(\theta)\leq\P(X_2>1 ,X_3\geq\theta X_2,\dots,X_{n+1}\geq\theta X_n)=0
    \end{align*}
    for $n\geq 1$. Therefore, when conditioning on the event $\Omega_n$, we only consider starting points $y$ fulfilling $y>\frac{1}{\theta}$.

\begin{Lemma}[Theorem~\ref{maintheorem} (G)]
    For $y>\frac{1}{\theta}$, the limiting process $M'$ exists and is a Markov chain on $S$ with transition kernel
    \begin{align*}
    p'(x_1,\d x_2)&\coloneqq\mathds{1}_S(x_2,x_1)\frac{h(x_2)}{h(x_{1})}\frac{1}{\lambda}\frac{1}{a+1}\,\d x_2,~~~~~~~~x_1,x_2\in \bigg[\frac{1}{\theta},1\bigg],
\end{align*}
where $\lambda$ is the persistence exponent of $(p_n^{a,y}(\theta))$ and 
\begin{align*}
         h(x)=\begin{cases}
           E\left(\theta,\frac{-\theta x}{\lambda (a+1)}\right),&\frac{1}{\theta}\leq \theta x< 1,
            \\E\left(\theta,\frac{-1}{ \lambda \theta(a+1)}\right),&~ \theta x< \frac{1}{\theta},
        \end{cases}
    \end{align*}
where $\lambda:=\tilde \lambda \, \frac{1-\frac{1}{\theta}}{a+1}$ and $\tilde\lambda$ as defined in \eqref{root_blue} for $a=\frac{-1}{\theta}$.
\end{Lemma}
\begin{proof}
   Analogously to the proof of Lemma~\ref{generatingfunction_grune_region}, set $A\coloneqq \{X_i\geq\frac{1}{\theta}~\forall~2\leq i\leq n+1\}$ and recall that the conditions $\theta X_i\leq X_{i+1}\leq 1$ for all $2\leq i\leq n$ and $X_{n+1}\geq\theta X_n\geq \theta\geq\frac{1}{\theta}$ imply $X_i\geq\frac{1}{\theta}$ for all $2\leq i \leq n$. In particular, $\P(\Omega_n\cap A^c)=0$.
    Moreover, define $\tilde{M}_i\coloneqq(\tilde X_{i+1},\tilde X_i)$ for i.i.d. $\tilde{X}_i \sim U\left[\frac{1}{\theta},1\right]$ for $i\geq 2$, $\tilde X_1:=y$, and $\tilde\Omega_n\coloneqq\{\tilde M_1\in S,\dots, \tilde M_n\in S\}$, then we have
    \begin{align*}
        &\P_{y}((M_k)_{1\leq k\leq m}\in\cdot~|\Omega_n)
        =\P_{y}((M_k)_{1\leq k\leq m}\in \cdot~|\Omega_n\cap A)
        =\P_{y}((\tilde M_k)_{1\leq k\leq m}\in\cdot~|\tilde\Omega_n).
    \end{align*}
    Since $\tilde{X}_i \sim U\left[\frac{1}{\theta},1\right]$ and $\frac{1}{\theta}< y\leq 1$ is valid for the blue region, Theorem~\ref{maintheorem} (B) provides that the limiting process $M'$ is a Markov chain with transition kernel
    \begin{align*}
    p'(x_1,\d x_2)&\coloneqq \mathds{1}_S(x_2,x_1)\frac{h(x_2)}{h(x_{1})}\frac{1}{\tilde\lambda}\frac{1}{1-\frac{1}{\theta}}\,\d x_2,~~~~~~~~x_1,x_2\in \left[\frac{1}{\theta},1\right],
\end{align*}
where $\tilde\lambda$ and $h$ are the eigenvalue and eigenfunction from the blue region in the special case $\tilde a=\frac{-1}{\theta}$, respectively given in \eqref{root_blue} and by 
 \begin{align*}
        h(x)=\begin{cases}
            0,~~&~\theta x \ge1,
            \\E\left(\theta,\frac{\theta^2 x}{\tilde\lambda (1-\theta)}\right),&\frac{1}{\theta}\leq \theta x< 1,
            \\E\left(\theta,\frac{1}{\tilde\lambda (1-\theta)}\right),&~ \theta x< \frac{1}{\theta}.
        \end{cases}
 \end{align*}
  Finally, Corollary \ref{pers_exp_green} completes the claim.
\end{proof}
\section{The yellow region}
The yellow region is given by $\theta\in[-1,0)$ and $a\in(0,-\theta)$. Following \cite{AurzadaRaschel}, we treat this region by using a relation between the generating functions of the green and yellow region. 
To simplify the calculation, we state two helpful lemmas before giving an explicit formula for the generating function in Lemma~\ref{generatingfunction_gelbe_region}. First, we require some definitions and initial facts.
\begin{definition}\label{def_q_n_and_Q}
    For any $\theta\in \R$ and $a>-1$, we define  $q_0^a(\theta)\coloneqq1$, $q_0^{a,y}(\theta)\coloneqq1$ and, for $n\geq 1$,
    $$
    q_n^a(\theta)\coloneqq\P(\theta X_1\geq  X_2,\dots,\theta X_n\geq X_{n+1}),\qquad q_n^{a,y}(\theta)\coloneqq\P(\theta y\geq  X_2,\dots,\theta X_n\geq X_{n+1}).
    $$ 
    The respective generating functions are denoted by
    $$\hat{Q}_{a,\theta}(z)\coloneqq\sum_{n=0}^{\infty} q_n^{a}(\theta)z^n ~~\text{ and }~~\hat{Q}_{a,y,\theta}(z)\coloneqq\sum_{n=0}^{\infty} q_n^{a,y}(\theta)z^n.$$
\end{definition}
\begin{Lemma}
For $\theta <0$ and $a>0$, we have, with the abbreviation $\tilde{y}\coloneqq\frac{-y}{a}$,
  \begin{align}\label{cor1_aurzada,raschel}
      q_n^a(\theta)=p_n^{1/a}(\theta)\qquad\text{and}\qquad q_n^{a,y}(\theta)=p_n^{\frac{1}{a},\tilde{y}}(\theta).
  \end{align}
\end{Lemma}
\begin{proof}
    The first statement is shown in Corollary 1 of \cite{AurzadaRaschel}. The second statement is obtained by the same method.
\end{proof}
The next lemma relates the generating function of the $(p_n^a(\theta))$ with those of the $(q_n^a(\theta))$ and $q_n^{a,y}(\theta)$.
\begin{Lemma}\label{zusammenhang_genfct_yellow_p_und_q}
  For $\theta\neq 0$ and $a>-1$, we have on the common domain of convergence    \begin{align}\label{eq_zusammenhang_genfct_yellow_p_und_q}
        \hat{P}_{a,y,\theta}(z)=\frac{\hat{Q}_{a,y,\theta}(-z)}{1-z\cdot\hat{Q}_{a,\theta}(-z)}.
    \end{align}
    \end{Lemma}
    \begin{proof}
        The idea of the proof follows the idea of Lemma 2 and Lemma 3 of \cite{AurzadaRaschel}. We use the notation from Definition~\ref{def_q_n_and_Q} and set $A_1\coloneqq\{\theta y\leq X_2\}$, $A_i\coloneqq\{\theta X_i\leq X_{i+1}\}$ for $2\le i\le n$ and $B_1\coloneqq\{\theta y\geq X_2\}$, $B_i\coloneqq\{\theta X_i\geq X_{i+1}\}$ for $2\leq i \leq n$. Note that due to the continuous distribution of the $X_i$, the set $B_i$ is the complement of the set $A_i$ up to a null set. Applying this to the set $A_n$, we get for $n\geq 1$
        \begin{align*}
            p_n^{a,y}(\theta)
            &=\P(A_1\cap\dots\cap A_n)
           =\P(A_1\cap\dots\cap A_{n-1})-\P(A_1\cap\dots\cap A_{n-1}\cap B_n)
            \\&=p_{n-1}^{a,y}(\theta)-\P(A_1\cap\dots\cap A_{n-1}\cap B_n).
        \end{align*}
        Using the fact that $B_{n-1}$ is the complement of $A_{n-1}$ up to a null set, we  split the probability $\P(A_1\cap\dots\cap A_{n-1}\cap B_n)$ and then use that $B_n$ is independent of  $A_1\cap\dots \cap A_{n-2}$:
        \begin{align*}
            \P(A_1\cap\dots\cap A_{n-1}\cap B_n)
            &=\P(A_1\cap\dots\cap A_{n-2}\cap B_n)-\P(A_1\cap\dots\cap A_{n-2}\cap B_{n-1}\cap B_n)
            \\&=\P(A_1\cap\dots\cap A_{n-2})\P( B_n)-\P(A_1\cap\dots\cap A_{n-2}\cap B_{n-1}\cap B_n)
            \\&=p_{n-2}^{a,y}(\theta)q_1^a(\theta)-\P(A_1\cap\dots\cap A_{n-2}\cap B_{n-1}\cap B_n).
        \end{align*}
        We therefore have
        \begin{align*}
            p_n^{a,y}(\theta)=p_{n-1}^{a,y}(\theta)-p_{n-2}^{a,y}(\theta)q_1^a(\theta)+\P(A_1\cap\dots\cap A_{n-2}\cap B_{n-1}\cap B_n).
        \end{align*}
        We proceed this way of transforming the terms of the form $\P(A_1\cap\dots\cap A_i\cap B_{i+1}\cap\dots\cap B_n)$. 
        In particular, for $n\geq 1$, we get the recursion
        \begin{align*}
            p_n^{a,y}(\theta)=\left(\sum_{k=1}^n(-1)^{k-1}p_{n-k}^{a,y}(\theta)q_{k-1}^a(\theta)\right)+(-1)^nq_n^{a,y}(\theta).
        \end{align*}
        Multiplying by $z^n$ and taking the sum over $n$, this becomes
        \begin{align*}
            \hat{P}_{a,y,\theta}(z)-1
            &=\sum_{n=1}^{\infty}p_n^{a,y}(\theta)z^n
            =\sum_{n=1}^{\infty}\sum_{k=1}^n(-1)^{k-1}p_{n-k}^{a,y}(\theta)q_{k-1}^a(\theta)z^n+\sum_{n=1}^{\infty}(-1)^nq_n^{a,y}(\theta)z^n
            \\&=\sum_{n=1}^{\infty}\sum_{k=1}^np_{n-k}^{a,y}(\theta)z^{n-k}zq_{k-1}^a(\theta)(-1)^{k-1}z^{k-1}+\sum_{n=1}^{\infty}(-1)^nq_n^{a,y}(\theta)z^n
            \\&=\sum_{k=1}^{\infty}\sum_{n=k}^{\infty}p_{n-k}^{a,y}(\theta)z^{n-k}zq_{k-1}^a(\theta)(-z)^{k-1}+\hat{Q}_{a,y,\theta}(-z)-1
            \\&=\sum_{k=1}^{\infty}zq_{k-1}^a(\theta)(-z)^{k-1}\sum_{n=k}^{\infty}p_{n-k}^{a,y}(\theta)z^{n-k}+\hat{Q}_{a,y,\theta}(-z)-1
            \\&=\hat{P}_{a,y,\theta}(z)\sum_{k=1}^{\infty}zq_{k-1}^a(\theta)(-z)^{k-1}+\hat{Q}_{a,y,\theta}(-z)-1
            \\&=\hat{P}_{a,y,\theta}(z)z\hat{Q}_{a,\theta}(-z)+\hat{Q}_{a,y,\theta}(-z)-1.
        \end{align*}
        Rearranging and dividing by  $1-z\hat{Q}_{a,\theta}(-z)$ gives the claim. 
    \end{proof}
From the last two lemmas, we infer the following relation. Note that in contrast to Lemma~\ref{zusammenhang_genfct_yellow_p_und_q}, the following result only holds for $\theta <0$ and $a>0$.
\begin{Corollary}
For $\theta <0$ and $a>0$, we have with the abbreviation $\tilde{y}\coloneqq\frac{-y}{a}$ on the common domain of convergence       \begin{align}\label{generatingfunction_yellow_and_green_rel}
        \hat{P}_{a,y,\theta}(z)=\frac{\hat{P}_{\frac{1}{a},\tilde{y},\theta}(-z)}{1-z\cdot\hat{P}_{\frac{1}{a},\theta}(-z)}.
    \end{align}
\end{Corollary}
\begin{proof}
    The statement follows directly by using \eqref{cor1_aurzada,raschel}
    to rewrite the terms on the right-hand side of \eqref{eq_zusammenhang_genfct_yellow_p_und_q}.
\end{proof}
Now we combine this with Lemma~\ref{generatingfunction_grune_region} to find an explicit representation of the generating function. 
\begin{Lemma}\label{generatingfunction_gelbe_region}
    In the yellow region, on the common domain of convergence, we have
    \begin{align*}
        \hat{P}_{a,y,\theta}(z)
        &=\begin{cases}
           \frac{E\left(\theta,\frac{-\theta zy}{a+1}\right)}{E\left(\theta,\frac{za}{\theta(a+1)}\right)-z\cdot\frac{ \theta +a}{\theta(a+1)}E\left(\theta,\frac{za}{a+1}\right)},~~&~-a\leq y \leq \frac{-a}{\theta},
            \\\frac{E\left(\theta,\frac{za}{a+1}\right)}{E\left(\theta,\frac{za}{\theta(a+1)}\right)-z\cdot\frac{\theta +a}{\theta(a+1)}E\left(\theta,\frac{za}{a+1}\right)},~~&~\frac{-a}{\theta}<y\leq1.
        \end{cases}    
    \end{align*}
\end{Lemma}

We stress that the representation in the last lemma is valid only up to the radius of convergence of the series on the left-hand side, which will be determined in Lemma~\ref{lemma_root_yellow}.

\begin{proof}
    Note that for $\theta\in[-1,0)$ and $a\in(0,-\theta)$ from the yellow region, we have that the pair $\theta$ and $\tilde{a}\coloneqq\frac{1}{a}>\frac{-1}{\theta}$ is part of the green region. Moreover, if $y\in[-a,1]$, it holds that $\tilde{y}\coloneqq\frac{-y}{a}\in[-\tilde{a},1]$. So we already calculated the value of $\hat{P}_{\frac{1}{a},\tilde{y},\theta}(-z)$ in Lemma~\ref{generatingfunction_grune_region} in the green region, which we recall to be
    \begin{align*}
        \hat{P}_{\frac{1}{a},\tilde{y},\theta}(-z)
        = \begin{cases}
            1,~~&~-\frac{1}{a}\leq \tilde{y}<\frac{1}{\theta},
            \\\frac{E\left(\theta,\frac{\theta z\tilde{y}a}{a+1}\right)}{E\left(\theta,\frac{za}{a+1}\right)},~~&~\frac{1}{\theta}\leq \tilde{y} \leq 1.
        \end{cases} 
    \end{align*}
   Furthermore, the value of $\hat{P}_{\frac{1}{a},\theta}(-z)$ can be taken from Lemma 7 of the green region in \cite{AurzadaRaschel}:
    \begin{align*}
        \hat{P}_{\frac{1}{a},\theta}(-z)
        &=\frac{\theta +a}{\theta(a+1)}+\frac{E\left(\theta,\frac{za}{\theta(a+1)}\right)-E\left(\theta,\frac{za}{a+1}\right)}{-zE\left(\theta,\frac{za}{a+1}\right)}.
    \end{align*}
     Plugging the last two displays (using $\tilde y=\frac{-y}{a}$) into \eqref{generatingfunction_yellow_and_green_rel} gives us the two claimed formulas.
\end{proof}

The next step is to determine the radius of convergence of $\hat P_{a,y,\theta}$ in Lemma~\ref{generatingfunction_gelbe_region}. This radius is governed by the zero of smallest modulus of the denominator appearing there. We therefore consider the equation
\begin{equation}\label{root_yellow_2}
\frac{E\left(\theta,\frac{za}{\theta(a+1)}\right)}{E\left(\theta,\frac{za}{a+1}\right)} = z \,\frac{\theta+a}{\theta(a+1)},\qquad \text{for $z\in\C$}.
\end{equation}

Let $z_0\in\C$ be the solution of (\ref{root_yellow_2}) with smallest modulus.

\begin{Lemma}\label{lemma_root_yellow}
We have $z_0\in\R$, $z_0>0$, for any other solution $z_1\in\C$ of (\ref{root_yellow_2}) we have $|z_1|>z_0$, and $E(\theta,\frac{za}{a+1})\neq 0$ for all $|z|\leq z_0$.
\end{Lemma}

\begin{proof}
Let us find a candidate for $z_0$ and then show that it is the solution of (\ref{root_yellow_2}) with smallest modulus. For this purpose, we first consider the real solutions of (\ref{root_yellow_2}). Define $m_0>0$ to be the solution with smallest modulus of the equation $E(\theta,\frac{-za}{a+1})=0$. This is a well-defined positive real number due to Lemma~\ref{nullstellen_exponentialfunktion}. Consider the function $z\mapsto \frac{E\left(\theta,\frac{za}{\theta(a+1)}\right)}{E\left(\theta,\frac{za}{a+1}\right)}$ for real $z$, which is the left-hand side of (\ref{root_yellow_2}). This function has a pole at $-m_0<0$, it has a zero at $-\theta m_0\in(0,m_0]$, its value at $0$ is $1$, and on $(-m_0,m_0)$ it has no other zeros; all of this follows by Lemma~\ref{nullstellen_exponentialfunktion}. By continuity, this means that it attains only positive values on $(-m_0,-\theta m_0)$. On the other hand, the right hand side of (\ref{root_yellow_2}) is a linear function with positive slope $\frac{\theta+a}{\theta(a+1)}$, which must have an intersection with the left-hand side for positive arguments. Therefore, there exists a smallest solution $z_0\in(0,-\theta m_0)$. 

Let us turn to solutions $z\in\C$ of (\ref{root_yellow_2}). We know, by Lemma~\ref{nullstellen_exponentialfunktion}, that $E(\theta,\frac{za}{\theta(a+1)})\neq 0$ for $z\in\C$ with $|z|<-\theta m_0$ and likewise $E(\theta,\frac{za}{a+1})\neq 0$ even for $z\in\C$ with $|z|<m_0$. This will imply the last part of the lemma's claim. Further, it shows that the (complex) solutions of (\ref{root_yellow_2}) for $|z|<m_0$ are the same as the (complex) zeros of the function
$$
g(z) :=  1 - z \,\frac{\theta+a}{\theta(a+1)} \cdot \frac{E\left(\theta,\frac{za}{a+1}\right)}{E\left(\theta,\frac{za}{\theta(a+1)}\right)}.
$$
It turns out that $g$ can be written with the help of a series with positive coefficients: Indeed, for $|-z/\theta|<m_0$, we have, by Lemma~\ref{generatingfunction_grune_region},
$$
g(z) =  1 - z \,\frac{\theta+a}{\theta(a+1)} \cdot \hat P_{\frac{1}{a},1,\theta}\left(\frac{-z}{\theta}\right) =  1 + \frac{\theta+a}{ a+1} \cdot \sum_{n=0}^\infty p_n^{\frac{1}{a},1}(\theta) (-1/\theta)^{n+1} z^{n+1} =: 1 - \tilde g(z),
$$
and the coefficients of the series of $\tilde g$, namely the numbers $\alpha_n:=-\frac{\theta+a}{a+1} p_n^{\frac{1}{a},1}(\theta) (-1/\theta)^{n+1}$, are all strictly positive. Note that the solutions to (\ref{root_yellow_2}) with $|z|<m_0$ correspond to the solutions of $g(z)=0$, i.e.\ $\tilde g(z)=1$.

We know from the construction above that $g(z_0)=0$, i.e.\ $\tilde g(z_0)=1$ and $0<z_0<m_0$. Now, let $z\in\C$ with $|z|<z_0$. Then, using positive coefficients,
$$
|\tilde g(z)| \leq \tilde g(|z|) < \tilde g(z_0) = 1.
$$
This shows that there cannot be any solution of $\tilde g(z_1)=1$ with $z_1\in\C$ and $|z_1|<z_0$. It remains to be seen that there are no solutions $z_1\in\C$ with $|z_1|=z_0$ either. Assume that $z_1=e^{i\phi } z_0$ is such a solution. Then
$$
1=|\tilde g(z_1)| \leq \tilde g(|z_1|) = \tilde g(z_0) =1, 
$$
and the triangle inequality step can only be an equality if all the terms in the sum have the same complex argument, i.e.\ $\arg(\alpha_n z_1^{n+1})=\arg(\alpha_m z_1^{m+1})$ for all $n,m$. Since the coefficients $\alpha_n$ are positive, this is true if and only if the $z_1^n$ have the same complex argument for all $n\geq 1$, i.e.\ $e^{i\phi n}$ is the same for all $n$, which can only happen for $\phi=0$, i.e.\ for $z_1=z_0$. This shows that there are no solutions $z_1\in \C$ of  (\ref{root_yellow_2}) with $|z_1|\leq z_0$, $z_1\neq z_0$.
\end{proof}

Using Lemma~\ref{asymptotic_powerseries}, we deduce the following asymptotics:
\begin{Lemma}\label{asymptotic_gelb_eigenfunction}
    In the yellow region, it holds with $\lambda=\lambda(\theta,a)$ as given in \eqref{root_yellow_2} that
    \begin{align*}
       p_n^{a,y}(\theta)\sim c(\theta,a)E\left(\theta,\frac{-\theta y}{\lambda(a+1)}\right)\lambda^{n+1}
   \end{align*}
   for $-a\leq y \leq \frac{-a}{\theta}$; and 
   \begin{align*}
       p_n^{a,y}(\theta)\sim c(\theta,a)E\left(\theta,\frac{a}{\lambda(a+1)}\right)\lambda^{n+1}
   \end{align*}
   for $\frac{-a}{\theta}<y\leq1$, with the abbreviation
    \begin{align*}
       c(\theta, a)\coloneqq\frac{1}{\frac{-a}{\theta(a+1)}\left(E\left(\theta,\frac{a}{\lambda(a+1)}\right)\frac{-\theta}{a}- \frac{\theta +a}{\lambda(a+1)}
        E\left(\theta,\frac{\theta a}{\lambda(a+1)}\right)\right)} 
   \end{align*}
   \end{Lemma}
\begin{proof}
Similarly to the proof of Lemma~\ref{asymptotic_blau_erste}, we construct a function $f(z)$ with
    \begin{align}\label{gelb_f}
        \sum_{n=0}^{\infty}p_n^{a,y}(\theta)z^n=\frac{f(z)}{z_0-z},~~~~\text{for all}~|z|<z_0,
    \end{align}
   which is analytic on $|z|<z_1$ with some $z_0<z_1$ and $f\! \left(z_0\right)\neq 0$. Applying Lemma~\ref{asymptotic_powerseries} will then yield the claim.
\\Recall that the generating function in Lemma~\ref{generatingfunction_gelbe_region} has the structure 
\begin{align*}
    \hat{P}_{a,y,\theta}(z)=\frac{u(z)}{g(z)}
\end{align*}
with $g(z)\coloneqq E\left(\theta,\frac{za}{\theta(a+1)}\right)-z\frac{\theta +a}{\theta(a+1)}E\left(\theta,\frac{za}{a+1}\right)$ and $u(z)$ depending on the choice of $y$: 
\begin{align*}
    u(z)\coloneqq\begin{cases}
         E\left(\theta,\frac{-\theta zy}{a+1}\right),~~~&-a\leq y \leq \frac{-a}{\theta},
         \\E\left(\theta,\frac{za}{a+1}\right),~~&\frac{-a}{\theta}<y\leq1.
    \end{cases}
\end{align*}
Moreover, $z_0>0$ is the root with the smallest absolute value of $g$. Set $v(z)\coloneqq g(z)\frac{1}{z_0-z}$.
We show that the function $f$ defined by $f(z)\coloneqq\frac{u(z)}{v(z)}$ satisfies the assumptions stated above.

By Lemma~\ref{lemma_root_yellow}, $z_0$ is the unique zero of $g$ of smallest modulus with all other zeros lying outside $|z|<z_1$ for some $z_1>z_0$. We will see below that $g'(z_0)<0$, implying that $z_0$ is a simple zero. Therefore, $v$ is analytic and does not vanish on a disk $|z|<z_1$ with $z_1>z_0$. Further, $u$ is entire so that $f$ is analytic on $|z|<z_1$. It remains to be seen that $f(z_0)\neq 0$ in order to apply Lemma~\ref{asymptotic_powerseries} and to compute $f(z_0)$.

For this purpose, note that $v(z_0)=-g'(z_0)$, which can be computed readily:
\begin{align*}
    g'(z_0) & = E\left(\theta,\frac{z_0 a}{a+1}\right) \frac{a}{\theta(a+1)} - \frac{\theta +a}{\theta(a+1)}E\left(\theta,\frac{z_0 a}{a+1}\right) - z_0 \frac{\theta +a}{\theta(a+1)}E\left(\theta,\frac{\theta z_0a}{a+1}\right) \frac{a}{a+1} 
    \\
    & = \frac{-1}{a+1} E\left(\theta,\frac{z_0 a}{a+1}\right)  - z_0 \frac{a(\theta +a)}{\theta(a+1)^2}E\left(\theta,\frac{\theta z_0a}{a+1}\right) < 0,
\end{align*}
because the $E$-terms are positive (due to Lemma~\ref{lemma_root_yellow}) while the prefactors are negative since $\theta\in[-1,0)$ and $a<-\theta$.

It remains to show $u(z_0)\neq 0$ with the help of Lemma~\ref{lemma_root_yellow}. For  $-a\leq y \leq \frac{-a}{\theta}$, we have $u(z)= E\left(\theta,\frac{-\theta zy}{a+1}\right)$ and $|\theta yz_0|\le z_0 a$ and thus,
$u(z_0)= E\left(\theta,\frac{-\theta z_0y}{a+1}\right)\neq 0$. 
  For $\frac{-a}{\theta}<y\leq1$, 
  $u(z)= E\left(\theta,\frac{za}{a+1}\right)$ and
      $u(z_0)= E\left(\theta,\frac{z_0a}{a+1}\right)
      \neq 0$.
   This gives us $f(z_0)\neq 0$ as well as the value of $f(z_0)=f(\frac{1}{\lambda})$ in both of the cases and hence Lemma~\ref{asymptotic_powerseries} yields the claims. 
   \end{proof}
Now we can deduce Theorem~\ref{maintheorem} for the yellow region. 
\begin{proof}[Proof of Theorem~\ref{maintheorem} (Y)]
There are three sets to be considered: For $-a\leq \theta x\leq -a \theta$ (i.e.\ $-a\leq x\leq -a/\theta$), we know from Lemma~\ref{asymptotic_gelb_eigenfunction} that $p_n^{a,x}(\theta)\sim c(\theta,a)h(x)\lambda^{n+1}$ with \linebreak $h(x)= E\left(\theta,\frac{-\theta x}{\lambda(a+1)}\right)$.
Further, for $\theta\leq \theta x \leq -a$ (i.e.\ $-a/\theta\leq x\leq 1$), we also know from Lemma~\ref{asymptotic_gelb_eigenfunction} that $p_n^{a,x}(\theta)\sim c(\theta,a)h(x)\lambda^{n+1}$ with $h(x)= E\left(\theta,\frac{a}{\lambda(a+1)}\right)$.
These two asymptotics are the input of Lemma~\ref{Prob=Lim_int_trans_kernel} which shows the convergence claim to the Doob transformed process in Theorem~\ref{maintheorem} (Y). 

It is left to prove the eigenvalue equation \eqref{eq_eigenfunction} for $-a\theta<\theta x\leq 1$ (i.e.\ $1/\theta\leq x < -a$, which is not needed for the Doob limit claim). Here, we directly verify the eigenvalue equation. We use $-a\leq -a\theta\leq\theta x\leq 1$ as well as $x<0$ to see that for $-a\theta\leq\theta x\leq\frac{-a}{\theta}$, we have
      \begin{align*}
        &\int_{\{y>\theta x\}\cap[-a,1]} h(y) \frac{1}{a+1}\,\d y
        \\&=\frac{1}{a+1}\int_{\theta x}^{\frac{-a}{\theta}}h(y)\d y+\frac{1}{a+1}\int_{\frac{-a}{\theta}}^1h(y)\d y
        \\&=\frac{1}{a+1}\int_{\theta x}^{\frac{-a}{\theta}}E\left(\theta,\frac{-\theta y}{\lambda(a+1)}\right)\d y+\frac{1}{a+1}\int_{\frac{-a}{\theta}}^1E\left(\theta,\frac{a}{\lambda(a+1)}\right)\d y
        \\&=\lambda E\left(\theta,\frac{-\theta x}{\lambda(a+1)}\right)-\lambda E\left(\theta,\frac{a}{\theta\lambda(a+1)}\right)+\frac{\theta+a}{\theta(a+1)}E\left(\theta,\frac{a}{\lambda(a+1)}\right)
        \\&=\lambda h(x),
      \end{align*}
      where we used \eqref{root_yellow_2} for $z_0=1/\lambda$ in the last step, and for $\frac{-a}{\theta}<\theta x\leq 1$,
      \begin{align*}
          &\int_{\{y>\theta x\}\cap[-a,1]} h(y) \frac{1}{a+1}\,\d y
        =\frac{1}{a+1}\int_{\theta x}^{1}h(y)\d y=\frac{1-\theta x}{a+1} E\left(\theta,\frac{a}{\lambda(a+1)}\right)=\lambda h(x). \qedhere
      \end{align*}
\end{proof} 
\section{The orange region}\label{orange_region}
In this region, we have $\theta\in(0,1)$ and $a\in[-\theta,0)$.
In contrast to the other regions, we define $\lambda$ as the inverse of the smallest positive root $z_0$ of the equation
\begin{align}\label{root_orange}
    zF(z)=1,
\end{align}
where
\begin{align*}
    F(z)\coloneqq F_{\theta,a}(z)\coloneqq\sum_{i=0}^p\frac{(-1)^i(\theta^i+a)^{i+1}}{(i+1)!\theta^{\frac{i(i+1)}{2}}(a+1)^{i+1}}\, z^i
    =\sum_{i=0}^{\infty}\frac{(-1)^i(\theta^i+a)_+^{i+1}}{(i+1)!\theta^{\frac{i(i+1)}{2}}(a+1)^{i+1}}\, z^i
\end{align*}
with $p\coloneqq p(\theta,a)\coloneqq\max\{n\in\N_0:\theta^n\geq-a\}$ and we will see that $\lambda$ is the persistence exponent. 

We know from Lemma 10 of \cite{AurzadaRaschel} that the function $F$ enters the generating function of the persistence probabilities without fixed starting point via
\begin{align} \label{eqn:orange:gfintegrated}
    \hat{P}_{a,\theta}(z)=\frac{F(z)}{1-zF(z)}.
\end{align}

One important step is to show that the radius of convergence of $\hat P_{a,\theta}$ (and subsequently of the $\hat P_{a,y,\theta}$) is given by $z_0$ defined above. Before this is assured, we prove all identities involving generating functions on the disk $|z|<1$ (which poses no problem, as for the involved $0\leq p_n^{a,y}(\theta)\leq 1$) and then extend to the common domain of convergence.

To begin with, consider the function
\begin{equation} \label{eqn:orange:defn.u}
u(z)\coloneqq u_{\theta,a}(z)\coloneqq 1-zF(z).
\end{equation}
Using the definition of $F$, we can consider this function for all parameters $\theta\in(0,1)$ and $a\in(-1,0)$. Note that the orange region requires $a\geq -\theta$, and only for such $\theta$ and $a$ the function $F$ and thus $u$ has a connection to the persistence probabilities in the orange region. Further, note that for $-1<a'<-\theta$, we have $p=0$ so that $u_{\theta,a'}(z)=1-z$.

We will need auxiliary steps in order to determine the true asymptotic rate of the $(p_n^{a,y}(\theta))$. We first deal with real solutions to the equation $u(z)=0$.

\begin{Lemma}\label{orange_F_monoton}
    For all $\theta\in(0,1)$ and $a\in(-1,0)$ we have $u'(z_0)< 0$, where $z_0=z_0(a)$ is the smallest positive real solution of $u(z)=0$.
\end{Lemma}

For the proof, we will need an auxiliary representation of $u_{\theta,a}$. The proof follows directly from the explicit formulas for $u$ (and thus for the derivative $u'$).

\begin{Lemma}
Let $\theta\in(0,1)$ and $-\theta < a < 0$. Then
\begin{equation} \label{eqn:orangerepresentationuprime}
u_{\theta,a}'(z) = - \frac{1}{a+1} \left[ u_{\theta,a/\theta}\left( \frac{\theta+a}{a+1} \, z\right)+a u_{\theta,a/\theta}\left( \frac{\theta+a}{\theta(a+1)} \, z\right)\right].
\end{equation}
\end{Lemma}

The next lemma is an estimate on the persistence exponents for different $a$. It will be a useful auxiliary result later on. The proof follows directly by conditioning the random variables to be $\geq -a'$.

\begin{Lemma} \label{lem:orangecomparisonofpes}
Let $a'<a$. Then
$$
 p_{n}^{a}(\theta)
   \geq p_n^{a'}(\theta) \, \left( \frac{a'+1}{a+1} \right)^{n+1}.
$$
In particular, denoting by $z_0(a)$ the radius of convergence of the generating function $\hat P_{a,\theta}(z)$, we have
$$
z_0(a) \leq z_0(a')\, \frac{a+1}{a'+1}.
$$
\end{Lemma}


We are now able to prove Lemma~\ref{orange_F_monoton}.
    
    \begin{proof}[Proof of Lemma~\ref{orange_F_monoton}]
    The proof follows by induction. We start with the case $p=0$, i.e.\ when $-1<a<-\theta$. Then $F(z)=1$ and $u(z)=1-z$ is strictly decreasing on $\R_{\ge0}$, in particular at its first zero $z_0=1$.

The induction hypothesis (the case $p-1$) is that $u_{\theta,a'}'(z')<0$ for all $0\leq z'\leq z_0(a')$ for all $a'$ with $\theta^{p-1}\geq -a'>\theta^p$. We now show the same for $p$, i.e.\ for $\theta^p \geq -a>\theta^{p+1}$ we claim that $u_{\theta,a}'(z)<0$ for all $0\leq z\leq z_0(a)$.

For ease of notation, set $a'\coloneqq a/\theta<a$ and
$$
\beta \coloneqq \frac{a+1}{a'+1} = \frac{\theta (a+1)}{\theta+a}>1.
$$
Let $0\leq z\leq z_0(a)$. Set $z'\coloneqq z/\beta$. Then, by Lemma~\ref{lem:orangecomparisonofpes}, we get $0\leq z'\leq z_0(a)/\beta\leq  z_0(a')$. Rewriting the identity (\ref{eqn:orangerepresentationuprime}) in the new notation, one obtains
\begin{equation} \label{eqn:orangecomparisoncrucial}
(a+1)(-u_{\theta,a}'(z)) = u_{\theta,a/\theta}\left( \theta \frac{z}{\beta} \right)+ a \cdot u_{\theta,a/\theta}\left( \frac{z}{\beta} \right) = u_{\theta,a'}\left(  \theta z'  \right)- (-a) u_{\theta,a'}\left( z' \right).
\end{equation} 
By the induction hypothesis, we know that since $0\leq z'\leq z_0(a')$, the function $u_{\theta,a'}$ is strictly decreasing. In particular, $u_{\theta,a'}\left(  \theta z'  \right)>u_{\theta,a'}\left( z' \right)\geq 0$ (exclude $z'=0$ here). Since $(-a)\in[0,1)$, this implies (including again $z'=0$)
$$
u_{\theta,a'}\left(  \theta z'  \right)>(-a)u_{\theta,a'}\left( z' \right).
$$
Inserting this into (\ref{eqn:orangecomparisoncrucial}) gives $u_{\theta,a}'(z)<0$ for all $0\leq z\leq z_0(a)$, in particular for $z=z_0(a)$.
    \end{proof}

    To compute the generating function, we partition the region of the possible starting point $y\in[-a,1]$ into the intervals 
$$A_m\coloneqq\bigg[\frac{-a}{\theta^m},\frac{-a}{\theta^{m+1}}\bigg),~~~0\le m\le p-1,~~~A_p\coloneqq\bigg[\frac{-a}{\theta^p},1\bigg].$$ 
Note that for $y\in A_m$ we have $m=\max\{n\in\N_0:\theta^ny\geq-a\}$ and that for all $1\le m\le p$, the case $y\in\big[\frac{-a}{\theta^m},\frac{-a}{\theta^{m+1}}\big)$ is equivalent to  $\theta y\in\big[\frac{-a}{\theta^{m-1}},\frac{-a}{\theta^{m}}\big)= A_{m-1}$. 
\\For $m\ge1$, we will use the following lemma to reduce the generating function in the region $A_m$ to the region $A_{m-1}$:
\begin{Lemma}\label{orange_GF_reduction}
    In the orange region, for $m\ge1$, it holds that
    \begin{align} \label{eqn:orange:dglviadifferentregions}
       \frac{\partial}{\partial y}\hat{P}_{a,y,\theta}(z)   
    =\frac{-\theta z}{a+1}\hat{P}_{a,\theta y,\theta}(z),~~~~y\in A_m,
    \end{align}
    and
    \begin{align}\label{continuous_orange}
    \hat{P}_{a,\frac{-a}{\theta^{m}},\theta}(z)=\lim_{y\nearrow\frac{-a}{\theta^{m}}}\hat{P}_{a,y,\theta}(z).
\end{align}
\end{Lemma}
\begin{proof}
Let $y\in A_m$. For $1\le m\le p$, we have $1\ge\theta y\ge\frac{-a}{\theta^{m-1}}\ge-a$. Therefore, for $n\ge 1$,
\begin{align*}
     p_{n}^{a,y}(\theta)
     &=\P(\theta y\leq  X_2,\dots,\theta X_n\leq X_{n+1}) \notag
    \\&=\frac{1}{(a+1)^n}\int_{\theta y}^1\bigg[\int_{-a}^1\mathds{1}_{\{\theta x_{2}\leq x_{3}\}}\dots
  \int_{-a}^1\mathds{1}_{\{\theta x_{n}\leq x_{n+1}\}}\d x_{n+1}\dots \d x_3\bigg]\d x_2\notag
  \\&=\frac{1}{a+1}\int_{\theta y}^1 p_{n-1}^{a,x_2}(\theta)\d x_2.\notag
 \end{align*}
  Multiplying with $z^n$ and summing over $n$, we get
 \begin{align}
     \hat{P}_{a,y,\theta}(z)
       &=1+\frac{z}{a+1}\int_{\theta y}^1\hat{P}_{a,x_2,\theta}(z)\d x_2\label{eqn:orange:integralequationhatp}.
 \end{align}
 Taking the derivative with respect to $y$ gives us the equation (\ref{eqn:orange:dglviadifferentregions}).
Note that for $y\in A_m$, the function $\hat{P}_{a,\theta y,\theta}(z)$ is the generating function in the region $A_{m-1}$ since $\theta y\in A_{m-1}$. Moreover, the function
$y\mapsto\hat{P}_{a,y,\theta}(z)$ is continuous on $y\in[-a,1]$, so in particular at $y=\frac{-a}{\theta^{m}}$, which implies \eqref{continuous_orange}.
\end{proof}

In order to represent the generating function $\hat P_{a,y,\theta}$, we will use the following notation: 
Define $H_k(x)\coloneqq \frac{x^k}{k!}$ and note that $H'_{k+1}=H_k$ and $H_k(\theta x)=\theta ^kH(x)$. Now we start by treating the cases $y\in A_0$ and $y\in A_1$ and use the ideas to deduce a structure for general $y\in A_m$.
\begin{Lemma}\label{orange_GF_m=1}
   It holds
    \begin{align*}
        \hat{P}_{a,y,\theta}(z)&=c_{0,0}(\theta,a,z) H_0(y),~~~~~y\in A_0,
\\
        \hat{P}_{a,y,\theta}(z)&=\sum_{i=0}^1c_{1,i}(\theta,a,z)H_i(y),~~~~~y\in A_1,
    \end{align*}
    with \begin{align*}
   c_{0,0}&\coloneqq c_{0,0}(\theta,a,z)\coloneqq \frac{1}{1-zF(z)},
        \\c_{1,0}&\coloneqq c_{1,0}(\theta,a,z)\coloneqq  c_{0,0}+\frac{ -az}{a+1}\,c_{0,0} \qquad\text{and}\qquad
        c_{1,1}\coloneqq c_{1,1}(\theta,a,z)\coloneqq  \frac{- z\theta}{a+1}\,c_{0,0}.
\end{align*}
\end{Lemma}
\begin{proof}
First, consider $y\in A_0$, which means that $\theta y<-a$. So for $n\geq 1$, we have
\begin{align*}
     p_{n}^{a,y}(\theta)
     &=\P(\theta y\leq  X_2,\theta X_2\leq X_3,\dots,\theta X_n\leq X_{n+1}) 
   = p_{n-1}^{a}(\theta),
\end{align*}
so $ \hat{P}_{a,y,\theta}(z)
    =1+\sum_{n=1}^{\infty}p_{n}^{a,y}(\theta)z^n
    =1+z\hat{P}_{a,\theta}(z)$.
The generating function $\hat{P}_{a,\theta}(z)$ can be taken from Lemma 10 of \cite{AurzadaRaschel} (cf.\ \eqref{eqn:orange:gfintegrated}) so that
\begin{align*}
    \hat{P}_{a,y,\theta}(z)=1+\frac{zF(z)}{1-zF(z)}=\frac{1}{1-zF(z)}=c_{0,0}.
\end{align*}
Using Lemma~\ref{orange_GF_reduction}, we now get a closed form in case $y\in A_1$: For $y\in A_1$, we have $\theta y\in A_0$ and thus,
\begin{align*}
    \hat{P}_{a,y,\theta}(z)=\frac{- z\theta}{a+1}c_{0,0} y+c_{0,0}+\frac{ -az}{a+1}c_{0,0},
\end{align*}
and the claim follows. 
\end{proof}
The previous lemma suggests the following general structure:
\begin{Corollary}\label{orange_GF_rekursive_coeff}
    In the orange region, for any $0\le m\le p$, it holds
    \begin{align*}
        \hat{P}_{a,y,\theta}(z)=\sum_{i=0}^mc_{m,i}(\theta,a,z)H_i(y),~~~~~y\in A_m,
    \end{align*}
   with radius of convergence $|z|<\frac{1}{\lambda}$ and recursive definitions
    \begin{align*}
   c_{0,0}&\coloneqq c_{0,0}(\theta,a,z)\coloneqq \frac{1}{1-zF(z)},
        \\c_{m,i}&\coloneqq c_{m,i}(\theta,a,z)\coloneqq \frac{-z\theta^i}{a+1}c_{m-1,i-1},~~~~~1\le i\le m,~~~~\text{and}
        \\c_{m,0}&\coloneqq c_{m,0}(\theta,a,z)\coloneqq c_{m,0}\coloneqq\sum_{i=0}^{m-1}\left(c_{m-1,i}H_i\left(\frac{-a}{\theta^m}\right)-c_{m,i+1}H_{i+1}\left(\frac{-a}{\theta^{m}}\right)\right).
\end{align*}
\end{Corollary}
    \begin{proof}
       We prove the corollary via induction. For $y\in A_0$ and $y\in A_1$, we have already verified the claim in Lemma~\ref{orange_GF_m=1}. 
       Assume that for $\tilde y\in A_{m-1}$, we have $\hat{P}_{a,\tilde y,\theta}(z)=\sum_{i=0}^{m-1}c_{m-1,i}H_i(\tilde y)$. In particular, this holds for $\tilde y=\theta y$ when $y\in A_m$.
       We show that  the function $y\mapsto\sum_{i=0}^mc_{m,i}H_i(y)$ on $A_m$ satisfies the claims of Lemma~\ref{orange_GF_reduction}. First, we check the differential equation:
       \begin{align*}
           \frac{\partial}{\partial y}\sum_{i=0}^mc_{m,i}H_i(y)
           &=\sum_{i=1}^mc_{m,i}H_{i-1}(y)
           =\sum_{i=1}^m\frac{-z\theta^ic_{m-1,i-1}}{a+1}H_{i-1}(y)
          \\&=\frac{-z\theta}{a+1}\sum_{i=1}^mc_{m-1,i-1}H_{i-1}(\theta y)
           =\frac{-z\theta}{a+1}\sum_{i=0}^{m-1}c_{m-1,i}H_{i}(\theta y)
           \\&=\frac{-z\theta}{a+1}\hat{P}_{a,\theta y,\theta}(z).
       \end{align*}
      To check continuity in $\frac{-a}{\theta^m}$, we compute $
          \lim_{\tilde y\nearrow\frac{-a}{\theta^{m}}}\hat{P}_{a,\tilde y,\theta}(z)$ with the obtained formula for $
          \hat{P}_{a,\tilde y,\theta}(z)$ and $\tilde y\in A_{m-1}$. A direct computation yields
          \begin{align*}
              \lim_{\tilde y\nearrow\frac{-a}{\theta^{m}}}\hat{P}_{a,\tilde y,\theta}(z)
              &=\sum_{i=0}^{m-1}c_{m-1,i}H_i\left(\frac{-a}{\theta^m}\right)
              \\&=\sum_{i=1}^mc_{m,i}H_i\left(\frac{-a}{\theta^{m}}\right)+\sum_{i=0}^{m-1}\left(c_{m-1,i}H_i\left(\frac{-a}{\theta^m}\right)-c_{m,i+1}H_{i+1}\left(\frac{-a}{\theta^{m}}\right)\right)
              \\&=\sum_{i=1}^mc_{m,i}H_i\left(\frac{-a}{\theta^{m}}\right)+c_{m,0}H_0\left(\frac{-a}{\theta^{m}}\right)
              \\&=\sum_{i=0}^mc_{m,i}H_i\left(\frac{-a}{\theta^{m}}\right),
          \end{align*}
          which is equal to $ \hat{P}_{a,\frac{-a}{\theta^m},\theta}(z)$. This completes the proof.
    \end{proof}
Now we iterate the coefficients from the last lemma to get a form only depending on the coefficients $c_{i,0}$ for $0\le i\le m$:
\begin{Corollary}\label{orange_GF_iterierte_coeff}
    In the orange region, for $0\le m\le p$, it holds
    \begin{align} \label{eqn:orange:hh}
       \hat{P}_{a,y,\theta}(z)=\sum_{i=0}^m\frac{\theta^\frac{i(i+1)}{2}}{i!}\left(\frac{-zy}{a+1}\right)^ic_{m-i,0},~~~~~y\in A_m,
    \end{align}
    where $c_{0,0}=\frac{1}{1-zF(z)}$ and \begin{align}
            c_{k,0}
            &=\sum_{i=0}^{k-1}\theta^{\frac{i(i+1)}{2}-ki}\left(\frac{za}{a+1}\right)^{i+1}\frac{1}{(i+1)!}\left(\frac{(a+1)(i+1)}{za}-\theta^{i+1-k}\right)c_{k-1-i,0},~~~k\ge 1.\label{orange_sum_c}
        \end{align}
    \end{Corollary}
    \begin{proof}
        By induction, we have for $m\ge 0$ and $0\le i \le m$
        \begin{align}\label{orange_ind_cmi}
            c_{m,i}
            =\left(\frac{-z}{a+1}\right)^i\theta^\frac{i(i+1)}{2}c_{m-i,0}.
        \end{align}
        Corollary~\ref{orange_GF_rekursive_coeff} therefore yields (\ref{eqn:orange:hh}). 
        Plugging \eqref{orange_ind_cmi} into the formula for $c_{k,0}$ for $1\le k\le m$ yields the second part: 
        \begin{align*}
            c_{k,0}
            &=\sum_{i=0}^{k-1}\left(\frac{-z}{a+1}\right)^i\theta^\frac{i(i+1)}{2}c_{k-1-i,0}H_i\left(\frac{-a}{\theta^k}\right)-\left(\frac{-z}{a+1}\right)^{i+1}\theta^\frac{(i+2)(i+1)}{2}c_{k-1-i,0}H_{i+1}\left(\frac{-a}{\theta^k}\right)   
            \\&=\sum_{i=0}^{k-1}\left(\frac{-z}{a+1}\right)^i\theta^\frac{i(i+1)}{2}c_{k-1-i,0}H_i\left(\frac{-a}{\theta^k}\right)\left(1-\frac{za\theta^{i+1-k}}{(a+1)(i+1)}\right)
            \\&=\sum_{i=0}^{k-1}\theta^{\frac{i(i+1)}{2}-ki}\left(\frac{za}{a+1}\right)^{i+1}\frac{1}{(i+1)!}\left(\frac{(a+1)(i+1)}{za}-\theta^{i+1-k}\right)c_{k-1-i,0}.\qedhere
        \end{align*}
    \end{proof}
Using the just obtained representation of the coefficients, we can eventually find a closed form of the generating function, isolating the dependence on the function $F$.
\begin{Corollary}\label{orange_GF_in_G_und_F}
    In the orange region, for $0\le m\le p$, we have  $$\hat{P}_{a,y,\theta}(z)=\frac{G_{a,y,\theta}(z)}{1-zF(z)}$$ where for $y\in A_m$, we set
    \begin{align*}
        G_{a,y,\theta}(z)\coloneqq\sum_{i=0}^m \left(\frac{-zy}{a+1}\right)^i\frac{\theta^\frac{i(i+1)}{2}}{i!}\sum_{{\ell}=0}^{m-i}\sum_{\substack{i_1+\dots +i_{\ell}= m -i-{\ell}\\i_r\ge 0}}\prod_{r=1}^{\ell}s_z(m-i-(r-1)-\sum_{j=1}^{r-1}i_j,i_r),
    \end{align*}
    with the convention that the sum on the right-hand side is equal to $1$ for ${\ell}=0$, $i=m$ and equal to $0$ for ${\ell}=0$, $i<m$, and
    \begin{align*}
        s_z(k,i)\coloneqq \theta^{\frac{i(i+1)}{2}-ki}\left(\frac{za}{a+1}\right)^{i+1}\frac{1}{(i+1)!}\left(\frac{(a+1)(i+1)}{za}-\theta^{i+1-k}\right).
    \end{align*}
    \end{Corollary}
    \begin{proof} Iterating \eqref{orange_sum_c}, we have for $k\ge 1$:
    \begin{align}
        c_{k,0}&=\sum_{i=0}^{k-1}s_z(k,i)c_{k-1-i,0}\notag
        \\&=\sum_{i_1=0}^{k-1}s_z(k,i_1)\sum_{i_2=0}^{k-2-i_1}s_z(k-1-i_1,i_2)c_{k-2-i_1-i_2,0}\notag
           \\&=\frac{1}{1-zF(z)}\sum_{{\ell}=1}^k\sum_{\substack{i_1+\dots +i_{\ell}= k-{\ell}\\i_r\ge 0}}\prod_{r=1}^{\ell}s_z(k-(r-1)-\sum_{j=1}^{r-1}i_j,i_r),\notag
    \end{align}
    where we summed over all possible decompositions $k=\ell+i_1+\dots+i_\ell$.
    Plugging this into Corollary~\ref{orange_GF_iterierte_coeff} gives
    \begin{align*}
       \hat{P}_{a,y,\theta}(z)
       &= \sum_{i=0}^m\frac{\theta^\frac{i(i+1)}{2}}{i!}\left(\frac{-zy}{a+1}\right)^ic_{m-i,0}
       \\&=\frac{1}{1\!-\!zF(z)}\!\sum_{i=0}^{m}\frac{\theta^\frac{i(i+1)}{2}}{i!}\left(\frac{-zy}{1\!+\!a}\right)^i\sum_{{\ell}=0}^{m-i}\sum_{\substack{i_1+\dots +i_{\ell}= m-i-{\ell}\\i_r\ge 0}}\prod_{r=1}^{\ell}s_z(m\!-\!i\!-\!(r\!-\!1)\!-\!\sum_{j=1}^{r-1}i_j,i_r). \qedhere
    \end{align*}
    \end{proof}

In the proofs of the asymptotic behaviour of  $(p_n^{a,y}(\theta))$, we shall need some more auxiliary results.

\begin{Lemma}\label{orange_aux_lemma_p_n} Let $0\leq m\leq p$ and $y\in A_m$. Then there exists $\kappa(y)>0$ such that, for all $n\geq m+1$, we have
$$
p_n^{a,y}(\theta) \geq \kappa(y) p_{n-m-1}^a(\theta).
$$
As a consequence,
$$
\hat P_{a,y,\theta}(z) 
\geq \kappa(y) z^{m+1} \hat P_{a,\theta}(z),\qquad z\geq 0.
$$
\end{Lemma}

\begin{proof} 
Let us start with the case $m=0$. Then $\theta y\leq -a$ and we have
$$
p_n^{a,y}(\theta) = p_{n-1}^a(\theta),
$$
so that the claim holds for $m=0$ with $\kappa(y)\coloneqq 1$. For $m\geq 1$, we proceed by induction. Let $y\in A_m$. Assume there exist $\alpha=\alpha(y)$ and $\beta=\beta(y)$ with
\begin{equation} \label{eqn:choicealphabeta}
-a\leq \alpha < \beta \leq 1, \qquad \alpha\geq \theta y, \qquad \beta \in A_{m-1}.
\end{equation}
Then, by the Markov property and the monotonicity of $x\mapsto p_{n-1}^{a,x}(\theta)$,
\begin{align*}
p_n^{a,y}(\theta)
 = & \, \int_{-a}^1 \1_{x_2\geq \theta y} \, p_{n-1}^{a,x_2}(\theta) \, \frac{1}{a+1}\d x_2
\geq  \, \int_{\alpha}^\beta \, p_{n-1}^{a,x_2}(\theta) \, \frac{1}{a+1}\d x_2
\\\geq & \, \int_{\alpha}^\beta \, p_{n-1}^{a,\beta}(\theta) \, \frac{1}{a+1}\d x_2
=  \, \frac{\beta-\alpha}{a+1}\, p_{n-1}^{a,\beta}(\theta).
\end{align*}
Since $\beta\in A_{m-1}$, we can apply the induction hypothesis and estimate 
$$
p_{n-1}^{a,\beta}(\theta)\geq \kappa(\beta) p_{(n-1)-(m-1)-1}^a(\theta)=\kappa(\beta) p_{n-m-1}^a(\theta),
$$
so that for $y\in A_m$ we can set $\kappa(y)\coloneqq \kappa(\beta(y)) \frac{\beta(y)-\alpha(y)}{a+1}$ and obtain the claim by combining the last two displays.

We still have to show that there are valid choices for $\alpha=\alpha(y)$ and $\beta=\beta(y)$ satisfying (\ref{eqn:choicealphabeta}): set 
$$
\alpha\coloneqq \theta y,\qquad \beta\coloneqq \frac{1}{2}\left( \theta y+\frac{-a}{\theta^m}\right).
$$
Then $\alpha=\theta y \geq -a$, since $m\geq 1$. Further,
$$
\beta - \alpha= \frac{1}{2}\left( \theta y+\frac{-a}{\theta^m}\right) - \theta y =\frac{1}{2}\left( -\theta y+\frac{-a}{\theta^m}\right) > \frac{1}{2}\left(-\theta \frac{-a}{\theta^{m+1}}+\frac{-a}{\theta^m}\right) = 0,
$$
which in particular means
$
\beta > \alpha = \theta y \geq \frac{-a}{\theta^{m-1}}
$.
Finally,
$$
\beta = \frac{1}{2}\left( \theta y+\frac{-a}{\theta^m}\right) < \frac{1}{2}\left( \theta \frac{-a}{\theta^{m+1}}+\frac{-a}{\theta^m}\right) = \frac{-a}{\theta^m} \leq 1,
$$
so that indeed $\beta\in A_{m-1}$. Note that in case $m=p$, we additionally used the fact that $\theta<\frac{-a}{\theta^p}$ in the inequalities above.
\end{proof}

A consequence from the last auxiliary result is the following statement about $G_{a,y,\theta}$ at $z_0$.

\begin{Lemma} \label{lem:gofz0positive}
    We have $G_{a,y,\theta}(z_0)>0$.
\end{Lemma}

\begin{proof}
Using Corollary~\ref{orange_GF_in_G_und_F}, Lemma~\ref{orange_aux_lemma_p_n} and (\ref{eqn:orange:gfintegrated}), we have for real $0\leq z<z_0$ and with  $\kappa(y)>0$ that
\begin{align*} 
\frac{G_{a,y,\theta}(z)}{u(z)}
& = \hat P_{a,y,\theta}(z)
 \geq \kappa(y) z^{m+1} \hat P_{a,\theta}(z)
 =\kappa(y) z^{m+1} \, \frac{F(z)}{u(z)},
\end{align*}
where $u$ is as in (\ref{eqn:orange:defn.u}).
Since $u(z)>0$ for all $0\leq z<z_0$ (recall that $z_0$ is the smallest positive real zero of $u$ and $u(0)=1$), we know that for all $0\leq z <z_0$
\begin{align*} 
 G_{a,y,\theta}(z)
& \geq  \kappa(y) z^{m+1}F(z),
\end{align*}
and letting $z\uparrow z_0$, we see that, using the definition of $z_0$,
\begin{align*} 
 G_{a,y,\theta}(z_0)
& \geq  \kappa(y) z_0^{m+1}F(z_0) =  \kappa(y) z_0^{m+1}\, \frac{1}{z_0} > 0. \qedhere
\end{align*}
 \end{proof}

Using the last auxiliary result, we can now show that the radius of convergence of $\hat P_{a,y,\theta}$ is indeed $z_0>0$ as defined in Lemma~\ref{orange_F_monoton}.

Recall that Corollary~\ref{orange_GF_in_G_und_F} states that
$\hat P_{a,y,\theta}(z) = \frac{G_{a,y,\theta}(z)}{u(z)},
$
where $G_{a,y,\theta}(z)$ is a polynomial in $z$. Recall from Lemma~\ref{orange_F_monoton} that there exists $z_0>0$ such that $u(z_0)=0$ and $z_0$ is a simple root of $u$. While Lemma~\ref{orange_F_monoton} is concerned with real solutions to the equation $zF(z)=1$, i.e.\ the real roots of $u(z)=0$, we will show that all other roots have modulus that is strictly larger than $z_0$.

\begin{Lemma} \label{lem:analyticityofremovedu}
Let $z_1\in\C$ with $|z_1|\leq z_0$ and $u(z_1)=0$. Then $z_1=z_0$. In particular, there exists an $\varepsilon>0$ such that for any $z_1\neq z_0$ that is a zero of $u(z)$ we have $|z_1|>z_0+\varepsilon$.
\end{Lemma}

\begin{proof}
Let $z_1\in\C$ with $|z_1|\leq z_0$ and $u(z_1)=0$. Assume $z_1\neq z_0$.

We start with some initial definitions. Recall that $u$ is a polynomial of degree $\leq p+1$ which has $z_0$ and $z_1$ as zeros. Denoting the multiplicity of the zero $z_1$ by $r$, we note that $r\in\{1,\ldots,p\}$. Further, we define
\begin{align} \label{eqn:defnphi}
\phi_1(y) & :=\lim_{z\to z_1} (z_1-z)^r \hat P_{a,y,\theta}(z) = \frac{G_{a,y,\theta}(z_1)}{\lim_{z\to z_1} \frac{u(z)}{(z_1-z)^r} },\qquad y\in[-a,1].
\end{align}
We note that $G_{a,y,\theta}(z_1)$ is some finite number (which may be zero), because $G_{a,y,\theta}(z)$ is a polynomial in $z$, while the limit in the denominator exists and is non-zero because $r$ is the exact multiplicity of $z_1$ as a zero of $u$.
Note that $\phi_1$ is not identically zero on $[-a,1]$. Indeed, at least for $y\in A_0$ we have
$$
\phi_1(y)=\lim_{z\to z_1} (z_1-z)^r \hat P_{a,y,\theta}(z) = \lim_{z\to z_1} (z_1-z)^r \frac{1}{u(z)},
$$
because $\hat P_{a,y,\theta}(z)=\frac{1}{u(z)}$; and the limit does not vanish, as argued above.

Define
\begin{align*} 
\phi_0(y):=\lim_{z\to z_0} (z_0-z) \hat P_{a,y,\theta}(z) =\frac{G_{a,y,\theta}(z_0)}{\lim_{z\to z_0} \frac{u(z)}{z_0-z} } = \frac{G_{a,y,\theta}(z_0)}{- \lim_{z\to z_0} \frac{u(z)-u(z_0)}{z-z_0} } =\frac{G_{a,y,\theta}(z_0)}{-u'(z_0)},
\end{align*}
and $G_{a,y,\theta}(z_0)$ is some finite number, as argued above.
Note that the last chain of equalities also shows that $\phi_0$ is strictly positive and finite on $[-a,1]$, by  Lemma~\ref{lem:gofz0positive} and Lemma~\ref{orange_F_monoton}.

We will now obtain eigenvalue equations for $\phi_1$ and $\phi_0$.
It is known from (\ref{eqn:orange:integralequationhatp}) that, for $y\in[-a,1]$, 
\begin{equation} \label{eqn:orange:recursion}
\hat{P}_{a,y,\theta}(z) =1+\frac{z}{a+1}\int_{\{x\geq \theta y\}\cap[-a,1]} \hat{P}_{a,x,\theta}(z)\d x.
\end{equation}
In fact, we showed this for $y\in A_m$, $m\geq 1$, only, but for $y\in A_0$ it is easy to verify.
We multiply (\ref{eqn:orange:recursion}) by $(z_1-z)^r$ and take the limit when $z\to z_1$. Then,
\begin{align} \label{eqn:orange:recursionphi}
\phi_1(y) =\frac{z_1}{a+1}\int_{\{x\geq \theta y\}\cap[-a,1]} \phi_1(x) \d x.
\end{align}
Here, we exchanged the limit and the integral. This is justified by the dominated convergence theorem, as visible from (\ref{eqn:defnphi}): the fraction $\frac{(z_1-z)^r}{u(z)}$ can be uniformly bounded when $z\to z_1$, because $z_1$ is a zero of $u$ of exact multiplicity $r$. Moreover, $G_{a,x,\theta}(z)$ is a polynomial in $z$ whose coefficients are uniformly bounded in $x\in[-a,1]$. Hence the integrand is dominated by an integrable constant, and dominated convergence applies. Equation (\ref{eqn:orange:recursionphi}) implies that $\phi_1$ is continuous.

Completely analogously (but multiplying this time (\ref{eqn:orange:recursion}) by $(z_0-z)$), we obtain
\begin{align} \label{eqn:orange:recursionh}
\phi_0(y) =\frac{z_0}{a+1}\int_{\{x\geq \theta y\}\cap[-a,1]} \phi_0(x) \d x.
\end{align}
The last equation shows that $\phi_0$ is in particular continuous. 
We can now iterate  (\ref{eqn:orange:recursionphi}) to get
\begin{align} 
\phi_1(y) & = \frac{z_1^n}{(a+1)^n} \int_{-a}^1 \int_{-a}^1 \cdots \int_{-a}^1 \1_{x_2\geq \theta y, x_3\geq \theta x_2, \ldots , x_{n+1} \geq \theta x_n } \d x_2 \ldots \d x_n \phi_1(x_{n+1}) \d x_{n+1}
\notag 
\\
& = z_1^n \int_{-a}^1 \P( X_2 \geq \theta y, X_3\geq \theta X_2, \ldots, X_n\geq \theta X_{n-1}, x_{n+1} \geq \theta X_n) \phi_1(x_{n+1})   \,\frac{\d x_{n+1}}{a+1}
\notag
\\
& =: z_1^n \int_{-a}^1 K_n(y,x_{n+1}) \phi_1(x_{n+1})  \,\frac{\d x_{n+1}}{a+1}.\label{eqn:orange:iteratedphi}
\end{align}
    Clearly,  $K_n(y,x_{n+1})>0$ for any $y,x_{n+1}\in[-a,1]$ if $n$ is large enough (in fact, $n\geq p+1$ already works).
By (\ref{eqn:orange:recursionh}), the analogous fact is true for $\phi_0$:
\begin{align} \label{eqn:orange:iteratedh}
\phi_0(y) & = z_0^n \int_{-a}^1 K_n(y,x_{n+1}) \phi_0(x_{n+1})  \,\frac{\d x_{n+1}}{a+1}.
\end{align}

We finally show some decisive estimates between $\phi_1$ and $\phi_0$. Since $\phi_0(y)>0$ for all $y\in[-a,1]$, we can define
$$
M:=\max_{y\in[-a,1]} \frac{|\phi_1(y)|}{\phi_0(y)}.
$$
Since $\phi_1$ and $\phi_0$ are continuous, the definition of $M$ allows to choose a $y_\ast\in[-a,1]$ such that
\begin{align*}
M = \frac{|\phi_1(y_\ast)|}{\phi_0(y_\ast)}.
\end{align*}
Using (\ref{eqn:orange:iteratedphi}) and (\ref{eqn:orange:iteratedh}) and the triangle inequality, for $n\geq p+1$, we obtain
\begin{align}
M \phi_0(y_\ast) & = |\phi_1(y_\ast)| \notag
\\
& = \left|z_1 ^n \int_{-a}^1 K_n(y_\ast,x_{n+1}) \phi_1(x_{n+1}) \,\frac{\d x_{n+1}}{a+1} \right| \notag 
\\
&  \leq |z_1|^n \int_{-a}^1 K_n(y_\ast,x_{n+1}) |\phi_1(x_{n+1})| \, \frac{\d x_{n+1}}{a+1} \label{eqn:orange:triangleintegral}
\\
& \leq z_0^n  \int_{-a}^1 K_n(y_\ast,x_{n+1}) M \phi_0(x_{n+1}) \, \frac{\d x_{n+1}}{a+1} \notag
\\
& = M \phi_0(y_\ast). \notag
\end{align}
Since $M$ and $\phi_0(y_\ast)$ are positive, the inequalities must in fact be equalities, so that in particular $|z_1|=z_0$. Further, since $K_n(y_\ast,x_{n+1})>0$ and the integrands $\phi_1(x_{n+1})$ and $ M \phi_0(x_{n+1})$ are continuous, the definition of $M$ and fact that
$$
\int_{-a}^1 K_n(y_\ast,x_{n+1}) (M \phi_0(x_{n+1})- |\phi_1(x_{n+1})|) \d x_{n+1} = 0
$$
imply that we must have $|\phi_1(x)|=M \phi_0(x)$ for all $x\in[-a,1]$. Additionally, the triangle inequality step (\ref{eqn:orange:triangleintegral}) can only be true if the complex argument of $\phi_1(x)$ is identical for Lebesgue almost all $x$; and by continuity of $\phi_1$, we must have $\phi_1(x)=e^{i\psi} M \phi_0(x)$ for some $\psi\in[0,2\pi)$.
Using this complex representation of $\phi_1(y)$ and (\ref{eqn:orange:recursionphi}) in the first and   (\ref{eqn:orange:recursionh}) in the last step, we obtain
\begin{align*}
e^{i\psi} M \phi_0(y) & = \frac{z_1}{a+1} \int_{\{x\geq \theta y\}\cap[-a,1]}  e^{i\psi} M \phi_0(x) \d x 
\\
& =  \frac{z_1}{z_0} \, e^{i\psi} M\, \frac{z_0}{a+1} \int_{\{x\geq \theta y\}\cap[-a,1]}  \phi_0(x) \d x =  \frac{z_1}{z_0} e^{i\psi} M\, \phi_0(y),
\end{align*}
which implies $z_1=z_0$. The second part of the assertion follows simply by the fact that $u$ is a polynomial, so that the different zeros are necessarily separated.
\end{proof}

As a next step, we prove the asymptotic behavior of $(p_n^{a,y}(\theta))$ for $y\in A_m$, $m\ge 1$.

\begin{Lemma}\label{orange_asymptotic_first}
    In the orange region, for $0\leq m\leq p$ and $y\in A_m$, it holds
    \begin{align*}
        p_n^{a,y}(\theta)\sim \lambda^{n+1}\sum_{i=0}^my^ic_i(\theta,\lambda,a, m)
    \end{align*}
    with \begin{align*}
        c_i(\theta,\lambda,a,m)\coloneqq\frac{-1}{u'\big(\frac{1}{\lambda}\big)}\left(\frac{-1}{\lambda(a+1)}\right)^i\frac{\theta^\frac{i(i+1)}{2}}{i!}\sum_{{\ell}=0}^{m-i}\sum_{\substack{i_1+\dots +i_{\ell}= m-i-{\ell}\\i_r\ge 0}}\prod_{r=1}^{\ell}s(m\!-\!i\!-\!(r\!-\!1)\!-\!\sum_{j=1}^{r-1}i_j,i_r),
    \end{align*}
    where $s(k,i)=s_{\frac{1}{\lambda}}(k,i)$ is given in \eqref{def_orange_s}.
\end{Lemma}
\begin{proof}
    Our goal is to use Lemma~\ref{asymptotic_powerseries}. By Corollary~\ref{orange_GF_in_G_und_F}, we have with $z_0=\frac{1}{\lambda}$
    \begin{align*}
        \hat{P}_{a,y,\theta}(z)&= \frac{G_{a,y,\theta}(z)}{1-zF(z)}=\frac{G_{a,y,\theta}(z)\frac{z_0-z}{1-zF(z)}}{z_0-z},
    \end{align*}
    and we have to show that for some $z_1>z_0$, the function $f(z)=G_{a,y,\theta}(z)\frac{z_0-z}{1-zF(z)}=\frac{G_{a,y,\theta}(z)(z_0-z)}{u(z)}$ is analytic on $|z|<z_1$ and $f(z_0)\neq 0$. 
    
The fact that $G_{a,y,\theta}(z)$ is entire follows directly from its definition. The fact that $z\mapsto \frac{z_0-z}{1-zF(z)}$ is analytic on $|z|<z_1$ for some $z_1>z_0$ is a conclusion from Lemma~\ref{orange_F_monoton} and Lemma~\ref{lem:analyticityofremovedu}.  Further, $G_{a,y,\theta}(z_0)\neq 0$ follows from Lemma~\ref{lem:gofz0positive}. 
Now, Lemma~\ref{asymptotic_powerseries} gives
    \begin{align*}
        p_n^{a,y}(\theta)&\sim f(z_0)\lambda^{n+1}\\&=\lambda^{n+1}\frac{-1}{u'(z_0)}\sum_{i=0}^m\left(\frac{-z_0y}{a+1}\right)^i\frac{\theta^\frac{i(i+1)}{2}}{i!}\sum_{{\ell}=0}^{m-i}\sum_{\substack{i_1+\dots +i_{\ell}= m-i-{\ell}\\i_r\ge 0}}\prod_{r=1}^{\ell}s(m-i-(r-1)-\sum_{j=1}^{r-1}i_j,i_r)
        \\&=\lambda^{n+1}\sum_{i=0}^my^ic_i(\theta,\lambda,a,m)
    \end{align*}
    with coefficients given in (\ref{eqn:orange:defnciwithlambda}).
\end{proof}
Now we can infer the claim of Theorem~\ref{maintheorem} for the orange region.
\begin{proof}[Proof of Theorem~\ref{maintheorem} (O)]
Lemma~\ref{orange_asymptotic_first} gives
   $p_n^{a,y}(\theta)\sim \lambda^{n+1}\sum_{i=0}^my^ic_i(\theta,\lambda,a, m)=\lambda^{n+1} h(y)$ for $0\le m\le p$ and $y\in A_m$ and $h$ as defined in the claim of Theorem~\ref{maintheorem} (O).   
Like in the other regions, we use Lemma~\ref{Prob=Lim_int_trans_kernel} to conclude to the Doob limit. The theorem also states the value of the eigenfunction for the case $\theta x\in A_p=[\frac{-a}{\theta^p},1]$. That part does not follow from the asymptotics and is not required for the application of Lemma~\ref{Prob=Lim_int_trans_kernel}. The value of $h$ for this case can be verified directly from the eigenvalue equation.
\end{proof}

\end{document}